\documentclass[11pt]{amsart}
\usepackage{amsmath, amsfonts, amsthm, amssymb}
\usepackage[all,cmtip,line]{xy}
\usepackage{array}
\usepackage{enumerate}
\usepackage{srcltx} 

\newtheorem{thm}{Theorem}[section]
\newtheorem{lem}{Lemma}[section]
\newtheorem{prop}[lem]{Proposition}
\newtheorem{cor}[lem]{Corollary}

\newtheorem{rem}[lem]{Remark}

\newtheorem{claim}[lem]{Claim}
\numberwithin{equation}{section}

\newtheorem{example}{Example}

\newcommand{\diam}{ \mbox{diam}}

\newcommand \eps{\varepsilon}

\newlength{\originalbase}
\newcommand{\spacing}[1]{\setlength{\baselineskip}{#1\originalbase}}

\begin{document}               

\newcommand{\avint}{{- \hspace{-3.5mm} \int}}

\spacing{1}

\title{The $L^{\infty}$ estimates for parabolic complex Monge-Ampere and Hessian equations}
\author{Xiuxiong Chen, Jingrui Cheng}
\maketitle

\centerline{Dedicated to Prof.  Lawson for his 80th birthday}

\begin{abstract}
In this paper,  we consider a version of parabolic complex Monge-Ampere equations,  and use a PDE approach similar to Phong et al to establish $L^{\infty}$ and H\"older estimates.  
We also generalize the $L^{\infty}$ estimates to parabolic Hessian equations.
\end{abstract}
\section{introduction}
This work tries to make the first step to develop a parabolic analogue of the uniform $L^{\infty}$ and H\"older continuity estimates for the complex Monge-Ampere and Hessian equations.

The question of deriving $L^{\infty}$ and H\"older estimates for the solution to the complex Monge-Ampere has been studied extensively in the last decades,  with minimal assumptions on the right hand side.  In the pioneering work of Kolodziej \cite{Kolo98},  he proved that if $e^F\in L^1(\log L)^p$ for some $p>n$,  then we have $L^{\infty}$ apriori bound for the solution $\varphi$.  He also showed that the solution is continuous.  Later on,  \cite{GKZ} proved the H\"older continuity of the solution to complex Monge-Ampere in the case of Dirichlet problem in $\mathbb{C}^n$.
In the compact setting without boundary,  \cite{DDGHKZ} proved the H\"older continuity of solutions to complex Monge-Ampere when the right hand side is in $L^p(\omega_0^n)$ for some $p>1$.
When it comes to Hessian equations, Dinew and Kolodziej in \cite{DK} derived the $L^{\infty}$ estimates for the $k-$Hessian equations on compact K\"ahler manifolds,  under some integrability assumption of the right hand side.

All the works above are done using methods from pluri-potential theory.  Since these works,  it has been a major question whether such results can be obtained by pure PDE method.  
This question has been answered for the case of bounded domain in $\mathbb{C}^n$,  and by Phong et al for the case of compact K\"ahler manifolds.

As to the flow problems,  the one that is considered most in the literature takes the following form:
\begin{equation}\label{1.1NNew}
\partial_t\varphi=\log(\frac{\omega_{\varphi}^n}{\omega_0^n})-F(t,z).
\end{equation}
There are many works concerning (\ref{1.1NNew}),  see for example \cite{EGZ},  \cite{EGZ2},  \cite{GLZ}.

In this work,  we propose to consider a different parabolic version for complex Monge-Ampere equation than (\ref{1.1NNew}),  for which we can prove the $L^{\infty}$ estimates very similar to the classical result by Kolodziej.  Moreover,  we show that it is possible to prove stability estimates very similar to \cite{GKZ} Theorem 1.1,  from which the H\"older continuity readily follows by an approximation technique,  developed in \cite{BD},  \cite{Demailly92},  \cite{Demailly}.  The approach we take to prove these results are purely PDE,  hence our work can be seen as a generalization of \cite{WWZ} and \cite{Phong}  to parabolic setting.

In order to motivate the form of the parabolic 
equation we will be considering,  we would like to go back to the real case.  Krylov-Tso (see \cite{Krylov} and \cite{Tso}) developed a parabolic version of Alexandrov maximum principle,  which says that:\\

Let $u$ be a function defined on $[0,T]\times \Omega$ where $\Omega$ is a bounded domain in $\mathbb{R}^n$,  then 
\begin{equation*}
\sup_{[0,T]\times \Omega}u\le \sup_{\partial_P([0,T]\times \Omega)}u+C_n(\diam \Omega)^{\frac{n}{n+1}}\bigg(\int_E|\partial_tu\det D_x^2u|dtdx\bigg)^{\frac{1}{n+1}}.
\end{equation*}
In the above,  $\partial_P([0,T]\times \Omega)$ is the parabolic boundary,  given by $(\{0\}\times \Omega)\cup([0,T]\times \partial\Omega)$,  and $E$ is the set of $(t,x)$ on which $\partial_tu\ge 0$ and $D_x^2u\le 0$.\\

The integrand on the right hand side is given by $\partial_tu\det(-D_x^2u)$,  hence it is very natural to take this product as our parabolic operator,  so that the parabolic (real) Monge-Ampere becomes:
\begin{equation*}
(-\partial_tu)\det(D_x^2u)=f\ge 0.
\end{equation*}
The admissible solutions we are considering
are that $\partial_tu\le0$ and $D_x^2u\ge 0$.

In the complex setting,  it naturally generalizes to $(-\partial_t\varphi)(\sqrt{-1}\partial\bar{\partial}\varphi)^n$,  and the class of admissible solutions are given by $\partial_t\varphi\le 0$,  $\sqrt{-1}\partial\bar{\partial}\varphi\ge 0$.
Hence in the context of compact K\"ahler manifold,  the equation reads:
\begin{equation}\label{1.1}
\begin{split}
&(-\partial_t\varphi)\omega_{\varphi}^n=e^F\omega_0^n,\\
&\varphi(0,\cdot)=\varphi_0.
\end{split}
\end{equation}

Since our goal is to derive apriori estimates,  we will assume that the solutions are all smooth,  so that our calculations are all justified.  But the bound we obtain only have the said dependence quantitatively.

In the above,  we assume that the initial data $\varphi_0$ is $\omega_0$-psh and is bounded,  and the admissible solutions we are looking for satisfies $\partial_t\varphi\le 0$,  $\omega_0+\sqrt{-1}\partial\bar{\partial}\varphi\ge 0$.
In order to study the convergence behavior of $\varphi$ as $t\rightarrow \infty$,  it is necessary to do a normalization first: $\tilde{\varphi}=\varphi-h(t)$.  Then the equation is transformed to:
\begin{equation*}
(h'(t)-\partial_t\tilde{\varphi})\omega_{\tilde{\varphi}}^n=e^F\omega_0^n.
\end{equation*}
Here the function $h(t)$ is so chosen in order to satisfy the volume compatibility condition:
\begin{equation*}
h'(t)\int_M\omega_0^n=\int_Me^{F(t,\cdot)}\omega_0^n.
\end{equation*}
We plan to investigate the convergence question as $t\rightarrow \infty$ in future works,  but for now,  we will only fix $0<T<\infty$ and derive estimates on $[0,T]\times M$.

We will establish the following result regarding the parabolic complex Monge-Ampere equation,  which can be seen as a parabolic analogue of Kolodziej's $C^0$ estimate:
\begin{thm}
Consider the equation (\ref{1.1}) on $[0,T]\times M$.  Assume that the right hand side has $L^1(\log L)^p$ integrability for some $p>n+1$.  In other words,  we assume that
\begin{equation*}
Ent_p(F):=\int_{[0,T]\times M}e^F(|F|^p+1)\omega_0^ndt<\infty.
\end{equation*}
We also assume that $\varphi_0$ is also uniformly bounded.  
Then $||\varphi||_{L^{\infty}}$ is uniformly bounded depending only on $||\varphi_0||_{L^{\infty}}$,  $T$,  $Ent_p(F)$,  the background metric,  $p$ and $n$.
\end{thm}

After this,  we consider the issue of H\"older contiuity of the solution when the initial value $\varphi_0\in C^{\bar{\alpha}}$.  We show that:
\begin{thm}\label{t1.2}
Let $\varphi$ solve (\ref{1.1}),  with $\varphi_0\in C^{\bar{\alpha}}$ for some $0<\bar{\alpha}<1$ and $e^F\in L^{p_0}(\omega_0^ndt)$ for some $p_0>1$.  Then for any $0\le s<t\le T$,  and any $\alpha<\frac{2}{1+q_0(n+1)}$,  
\begin{equation*}
|\varphi(t,x)-\varphi(s,x)|\le C(t-s)^{\frac{\alpha}{2}}.
\end{equation*}
For any $x,\,y\in M$,  and $t\in[0,T]$,  and any $\alpha\le \bar{\alpha}$,  $\alpha<\frac{2}{1+q_0(n+1)}$,
\begin{equation*}
|\varphi(t,x)-\varphi(t,y)|\le C|x-y|^{\alpha}.
\end{equation*}
Here $q_0=\frac{p_0}{p_0-1}$,  and the constant $C$ depends only on the background metric,  $n$,  $||e^F||_{L^{p_0}}$,  the $C^{\bar{\alpha}}$ norm of $\varphi_0$,  $T$ and choice of $\alpha<\frac{2}{1+q_0(n+1)}$.
\end{thm}

Next we will consider more general Hessian type equations,  in the form:
\begin{equation*}
f(-\partial_t\varphi,\lambda[h_{\varphi}])=e^F.
\end{equation*}
In the above $(h_{\varphi})_j^{i}=(\omega_0)^{i\bar{k}}(\omega_{\varphi})_{j\bar{k}}$ and $\lambda[h_{\varphi}]$ denotes the eigenvalues of $h_{\varphi}$ (which can be shown to be invariant under holomorphic coordinate change).  We also assume that $f$ is a $C^1$ function in terms of its variables.  The more precise structural assumptions on $f$ are set forth in the last section,  where we will also give some examples of $f$ satisfying our assumptions.

Under these assumptions,  we have the following $L^{\infty}$ estimates for the solution to the Hessian equation:
\begin{thm}
Let $\varphi$ be a solution to the equation $f(-\partial_t\varphi,\lambda[h_{\varphi}])=e^F$ on $[0,T]\times M$,  where $f$ satisfies the above structural assumption and $(-\partial_t\varphi,\lambda[h_{\varphi}])\in\Gamma$.  Assume also that for some $p>n+1$,  we have
\begin{equation*}
Ent_p(F):=\int_{[0,T]\times M}e^{(n+1)F}|F|^p\omega_0^ndt<\infty.
\end{equation*}
Then we can estimate $||\varphi||_{L^{\infty}}$ depending only on $||\varphi_0||_{L^{\infty}}$,  $T$,  the background metric,  $Ent_p(F)$ and $n$.
\end{thm}

The plan of the paper is as follows:\\

In section 2,  we will derive estimates for the parabolic complex Monge-Ampere equation,  including the $L^{\infty}$ estimates as well as H\"older estimates.

In section 3,  we generalize the $L^{\infty}$ estimates to more general Hessian equations.

{\bf Acknowlegement} The first-named author went through an exciting yet stressful period of time in the first half of last decade during which Prof. Lawson invited him to frequently drop by his office, and the wisdom he imparted in those wide-ranging conversations has had a great impact on the author. Such a kindness makes it an immense pleasure for us to dedicate this article to him, in celebration of his 80th birthday, in appreciation of his inspiring mathematical work, and in gratitude for his mentorship of young mathematicians.

\section{The parabolic complex Monge-Ampere equation}
In this section,  we consider the parabolic complex Monge-Ampere equation:
\begin{equation}\label{2.1NNN}
\begin{split}
&(-\partial_t\varphi)\omega_{\varphi}^n=e^F\omega_0^n,\\
&\varphi|_{t=0}=\varphi_0.
\end{split}
\end{equation}
Here we only consider solutions which are $\omega_0$-psh and that $\partial_t\varphi\le 0$.  The function $F(t,x)$ on the right hand side is given.  We also assume that the initial value $\varphi_0$ is bounded.

First we show that $\sup_M\varphi$ is bounded.
\begin{lem}\label{l2.1}
Assume that $\int_{[0,T]\times M}e^F\omega_0^ndt<\infty$,  then we have 
\begin{equation*}
|\sup_M\varphi|\le C.
\end{equation*}
Here $C$ depends on the background metric,  an upper bound for $\int_{[0,T]\times M}e^F\omega_0^ndt$ and $||\varphi_0||_{L^{\infty}}$.
\end{lem}
\begin{proof}
Since $\partial_t\varphi\le 0$,  we have
\begin{equation*}
\sup_{M}\varphi(t,\cdot)\le \sup_M\varphi_0\le ||\varphi_0||_{L^{\infty}}.
\end{equation*}
To estimate the lower bound for $\sup_M\varphi$,  we have to use the equation.  We consider the $I$-functional,  defined as:
\begin{equation*}
I(\varphi)=\frac{1}{n+1}\int_M\varphi\sum_{j=0}^n\omega_0^{n-j}\wedge\omega_{\varphi}^j.
\end{equation*}
To estimate the lower bound of $\sup_M\varphi$,  we will first get a lower bound for the $I$ functional,  then we will get a lower bound for $\int_M\varphi\omega_0^n$,  then we obtain a lower bound for $\sup_M\varphi$,  using some well-known arguments.

It is straightforward to find that
\begin{equation*}
\frac{d}{dt}I(\varphi)=\int_M\partial_t\varphi\omega_{\varphi}^n=-\int_Me^F\omega_0^n.
\end{equation*}
Therefore,  for any $t'\in[0,T]$,  we have that
\begin{equation*}
I(\varphi)=I(\varphi_0)+\int_0^{t'}\frac{d}{dt}I(\varphi)dt=I(\varphi_0)-\int_{[0,t']\times M}e^F\omega_0^ndt
\end{equation*}
Therefore,  we get that $I(\varphi)$ is bounded below on $t'\in[0,T]$,  with a lower bound having the said dependence.

Now we estimate a lower bound for $\int_M\varphi\omega_0^n$.  We can compute:
\begin{equation*}
\begin{split}
&\int_M\varphi\omega_0^n-I(\varphi)=\int_M\varphi\frac{1}{n+1}\sum_{j=0}^n\omega_0^{n-j}\wedge(\omega_0^j-\omega_{\varphi}^j)\\
&=\frac{1}{n+1}\int_M\varphi\sum_{j=0}^n\omega_0^{n-j}\wedge\sqrt{-1}\partial\bar{\partial}(-\varphi)\sum_{l=0}^{j-1}\omega_0^{j-1-l}\wedge\omega_{\varphi}^l\\
&=\frac{1}{n+1}\int_M\sqrt{-1}\partial\varphi\wedge \bar{\partial}\varphi\wedge\sum_{j=0}^n\sum_{l=0}^{j-1}\omega_0^{n-1-l}\wedge\omega_{\varphi}^l\ge 0.
\end{split}
\end{equation*}
So $\int_M\varphi\omega_0^n$ is bounded below as well.  On the other hand,  it is well known that for any $\omega_0$-psh function $\varphi$,  one has:
\begin{equation*}
|\frac{1}{vol(M)}\int_M\varphi\omega_0^n-\sup_M\varphi|\le C.
\end{equation*}
Here $C$ depends only on the background metric.  So we obtain that $\sup_M\varphi$ is also bounded from below.
\end{proof}
As a corollary,  we obtain that
\begin{cor}\label{c2.2}
There exists constant $\alpha_0>0$ depending only on the background metric,  such that
\begin{equation*}
\sup_{t\in[0,T]}\int_Me^{-\alpha_0\varphi}\omega_0^n\le C.
\end{equation*}
Here $C$ depends on the background metric,  an upper bound for $\int_{[0,T]\times M}e^F\omega_0^ndt$ and $||\varphi_0||_{L^{\infty}}$.
\end{cor}
\begin{proof}
Since $\varphi(t,\cdot)$ is $\omega_0$-psh for each $t\in[0,T]$,  we have that
\begin{equation*}
\sup_{t\in[0,T]}\int_Me^{-\alpha_0(\varphi-\sup_M\varphi)}\omega_0^n\le C.
\end{equation*}
From Lemma \ref{l2.1},  $\sup_M\varphi$ is uniformly bounded,  with the said dependence.  
\end{proof}
\subsection{Existence of smooth solution with smooth data}
Since later on,  we will need to use the solution to equation (\ref{2.1NNN}) as auxiliary functions for estimates,  we need to establish the solvability of (\ref{2.1NNN}) when the data is smooth.  More precisely,  we have:
\begin{prop}\label{p2.3}
Let $\varphi_0$ be a smooth function on $M$ with $\omega_0+\sqrt{-1}\partial\bar{\partial}\varphi_0>0$ and $F(t,x)$ is a smooth function on $[0,T]\times M$.  Then there exists a unique smooth solution to (\ref{2.1NNN}) on $[0,T]\times M$ starting from $\varphi_0$ such that $-\partial_t\varphi>0$ and $\omega_0+\sqrt{-1}\partial\bar{\partial}\varphi>0$.
\end{prop}
\begin{proof}
Uniqueness is quite easy to see,  thanks to the maximum principle.  Indeed,  if there are two such smooth solutions $\varphi$ and $\bar{\varphi}$,  we can consider $\varphi-\eps t-\bar{\varphi}$.  Assuming it has maximum at $(t_0,x_0)$ with $t_0>0$,  then we would have
\begin{equation*}
\partial_t(\varphi-\eps t-\bar{\varphi})|_{(t_0,x_0)}\ge 0,\,\,\sqrt{-1}\partial\bar{\partial}(\varphi-\eps t-\bar{\varphi})|_{(t_0,x_0)}\le 0.
\end{equation*}
This would imply
\begin{equation*}
\eps+(-\partial_t\varphi)\le -\partial_t\bar{\varphi},\,\,\,\omega_{\varphi}^n\le \omega_{\bar{\varphi}}^n.
\end{equation*}
Multiplying,  we get
\begin{equation*}
\eps\omega_{\varphi}^n+(-\partial_t\varphi)\omega_{\varphi}^n\le (-\partial_t\bar{\varphi})\omega_{\bar{\varphi}}^n.
\end{equation*}
This gives $\eps\omega_{\varphi}^n\le 0$,  which contradicts $\omega_{\varphi}>0$.  Hence $\varphi-\bar{\varphi}-\eps t\le (\varphi-\bar{\varphi}-\eps t)|_{t=0}=0$.  Letting $\eps\rightarrow 0$,  we get $\varphi\le \bar{\varphi}$.

It only remains to show existence,  and we can run a continuity method as follows: for $s\in[0,1]$,
\begin{equation}\label{2.2NNew}
\begin{split}
&(-\partial_t\varphi)\omega_{\varphi}^n=e^{sF}\omega_0^n,\\
&\varphi(0,\cdot)=s\varphi_0.
\end{split}
\end{equation}
When $s=0$,  it has a trivial solution $\varphi(t,x)=-t$.

To show openness,  we need to linearize the equation,  and the linearized operator is:
\begin{equation*}
Lu=-\partial_tu+(-\partial_t\varphi)\Delta_{\varphi}u.
\end{equation*}
Since $-\partial_t\varphi>0$,  the operator $L$ is uniformly parabolic.  Hence it would be standard to show the openness of the continuity path.

Now it only remains to show closeness of the continuity path,  for which we have to derive the apriori estimates.

First we would like to derive the equation for $\partial_t\varphi$.  Denote $v=\partial_t\varphi$,  we can differentiate (\ref{2.2NNew}) to get
\begin{equation*}
\begin{split}
&-\partial_tv+s\partial_tFv+(-\partial_t\varphi)\Delta_{\varphi}v=0,\\
&v|_{t=0}=-\frac{e^{sF(0,x)}\omega_0^n}{\omega_{\varphi_0}^n}.
\end{split}
\end{equation*}
Let $k>0$,  we define $\bar{v}=ve^{kt}$,  then in terms of $\bar{v}$,  the equation reads:
\begin{equation}\label{2.3NNew}
\begin{split}
&-\partial_t\bar{v}+(-\partial_t\varphi)\Delta_{\varphi}\bar{v}=\bar{v}(-k-s\partial_tF),\\
&-M_0\le \bar{v}|_{t=0}\le -\eps_0.
\end{split}
\end{equation}
In the above,  if we take $k>0$ such that $-k-s\partial_tF<0$,  then we see that $\bar{v}\le -\frac{1}{2}\eps_0$ on $[0,T]\times M$ by maximum principle.  In particular,  we get $v\le -\eps_1$ on $[0,T]\times M$.

On the other hand,  if $k<0$ such that $-k-s\partial_tF>0$,  then from (\ref{2.3NNew}),  we get that $-\partial_t\bar{v}+(-\partial_t\varphi)\Delta_{\varphi}v\le 0$,  so that $\bar{v}\ge -M_0$.

Hence from the above arguments,  we see that:
\begin{equation}
-M_1\le \partial_t\varphi\le -\eps_1.
\end{equation}
This implies that in particular,  $\varphi$ is uniformly bounded on $[0,T]\times M$.

Next we estimate the second derivative of $\varphi$.  Now we put $u=e^{-C\varphi}(n+\Delta \varphi)$ and denote $e^G=\frac{\omega_{\varphi}^n}{\omega_0^n}$,  then we can compute:
\begin{equation*}
\begin{split}
&-\partial_tu=(-C)(-\partial_t\varphi)u+e^{-C\varphi}\Delta(-\partial_t\varphi)=(-C)(-\partial_t\varphi)u+e^{-C\varphi}\Delta(e^{sF-G})\\
&\ge (-C)(-\partial_t\varphi)u+e^{-C\varphi} e^{sF-G}(s\Delta F-\Delta G)\\
&=(-\partial_t\varphi)(-Cu+e^{-C\varphi}s\Delta F-e^{-C\varphi}\Delta G).
\end{split}
\end{equation*}
On the other hand,  from Yau's calculation in \cite{Yau}:
\begin{equation*}
\Delta_{\varphi}u\ge e^{-C\varphi}(-C)\Delta_{\varphi}\varphi(n+\Delta\varphi)+e^{-C\varphi}\frac{R_{i\bar{i}k\bar{k}}(1+\varphi_{i\bar{i}})}{1+\varphi_{k\bar{k}}}+e^{-C\varphi}\Delta G-e^{-C\varphi}R.
\end{equation*}
In the above calculation,  we took normal coordinates at a point,  and $R_{i\bar{i}k\bar{k}}$ are curvature tensors of $\omega_0$,  $R$ is the scalar curvature of $\omega_0$.  Hence if we take $C$ large enough,  we would get that
\begin{equation*}
\begin{split}
&\Delta_{\varphi}u\ge e^{-C\varphi}tr_{\varphi}\omega_0(n+\Delta\varphi)-e^{-C\varphi}Cn(n+\Delta\varphi)+e^{-C\varphi}\Delta G-e^{-C\varphi}R\\
&=e^{-C\varphi}\big(tr_{\varphi}\omega_0 u-Cnu+\Delta G-R\big).
\end{split}
\end{equation*}
Hence if we define $L=-\partial_t+(-\partial_t\varphi)\Delta_{\varphi}$,  we get that
\begin{equation}\label{2.5N}
\begin{split}
&Lu\ge (-\partial_t\varphi)e^{-C\varphi}\big(-C(n+1)u-R+s\Delta F+tr_{\varphi}\omega_0u\big)\\
&\ge (-\partial_t\varphi)e^{-C\varphi}\big(-C(n+1)u-R+s\Delta F+e^{-\frac{G}{n-1}}(n+\Delta\varphi)^{\frac{1}{n-1}}u\big).
\end{split}
\end{equation}
In the above,  we note that $-\partial_t\varphi=e^{sF-G}$.  Also we have shown a bound for $-\partial_t\varphi$: $-M_1\le \partial_t\varphi\le -\eps_1$,  it follows that $G$ is bounded,  since $F$ is assumed to be smooth hence bounded.
Therefore,  it follows from (\ref{2.5N}) that 
\begin{equation*}
Lu\ge (-\partial_t\varphi)e^{-C\varphi}\big(-C(n+1)u-R+s\Delta F+cu^{\frac{n}{n-1}}\big).
\end{equation*}
Hence if $u$ achieves maximum at $(t_0,x_0)$,  then at $(t_0,x_0)$,  we would get
\begin{equation*}
0\ge -C(n+1)u-R+s\Delta F+cu^{\frac{n}{n-1}}.
\end{equation*}
This implies an upper bound for $u$ at $(t_0,x_0)$.
So we have shown that $\omega_{\varphi}\le C\omega_0$.  That $\omega_{\varphi}\ge \frac{1}{C}\omega_0$ follows from that $\frac{\omega_{\varphi}^n}{\omega_0^n}=e^G$ is bounded from above and below.
Hence the equation becomes uniformly parabolic and higher regularity follows from standard bootstrap.
\end{proof}

\subsection{Estimate the $L^{\infty}$ bound}
The first step is to establish a Moser-Trudinger type inequality,  similar to Lemma 1 in \cite{Phong}:
\begin{prop}\label{p2.3}
Denote $A_s=\int_{[0,T]\times M}(-\varphi-s)^+e^F\omega_0^ndt$.
Then there exists a constant $\beta_0>0$,  depending only on the background metric,  and there exists a constant $C>0$,  depending only on the background metric,  such that for any $s\ge ||\varphi_0||_{L^{\infty}}$
\begin{equation*}
\sup_{t\in[0,T]}\int_Me^{\beta_0A_s^{-\frac{1}{n+2}}((-\varphi-s)^+)^{\frac{n+2}{n+1}}}\omega_0^n\le C\exp(CE).
\end{equation*}
Here
\begin{equation*}
E=\int_{[0,T]\times M}(-\varphi)e^F\omega_0^ndt.
\end{equation*}
\end{prop}
To prove this,  we use the solution to an auxiliary problem.  Let $\eta_j:\mathbb{R}\rightarrow \mathbb{R}_+$ such that $\eta_j(x)\rightarrow \max(x,0)$ as $j\rightarrow \infty$ and $\eta_j(x)>x$ for $x>0$ (for example,  we could take $\eta_j(x)=\frac{1}{2}(x+\sqrt{x^2+j^{-1}})$.  We let $\psi_j$ be the solution to the following problem:
\begin{equation}\label{2.1NN}
\begin{split}
&(-\partial_t\psi_j)\omega_{\psi_j}^n=\frac{\eta_j(-\varphi-s)e^F\omega_0^n}{A_{j,s}},\\
&\psi_j|_{t=0}=0.
\end{split}
\end{equation}
Here 
\begin{equation*}
A_{j,s}=\int_{[0,T]\times M}\eta_j(-\varphi-s)e^F\omega_0^ndt.
\end{equation*}
The existence of such a $\psi_j$ is guaranteed by Proposition \ref{p2.3}.  
The above proposition will follow immediately once we prove:
\begin{lem}\label{l2.4}
There exists dimensional constant $c_n$ and $C_n$ such that
\begin{equation*}
c_nA_{j,s}^{-\frac{1}{n+1}}(-\varphi-s)^{\frac{n+2}{n+1}}\le -\psi_j+C_nA_{j,s}.
\end{equation*}
\end{lem}
\begin{proof}
Define the operator $L$ to be:
\begin{equation*}
Lu=-\partial_tu+(-\partial_t\varphi)\Delta_{\varphi}u.
\end{equation*}
Let $\eps=(\frac{1}{2}n^{\frac{n}{n+1}}\frac{n+1}{n+2})^{\frac{n+1}{n+2}}A_s^{-\frac{1}{n+2}}$,  $\Lambda=(\frac{1}{2}n^{\frac{n}{n+1}})^{-(n+1)}\frac{n+1}{n+2}A_s$,  we define $\Phi_j=\eps(-\varphi-s)-(-\psi_j+\Lambda)^{\frac{n+1}{n+2}}$.
Then we may compute (we suppress the subscript $j$ for the convenience of notations):
\begin{equation*}
\begin{split}
&L\Phi=-\partial_t\Phi+(-\partial_t\varphi)\Delta_{\varphi}\Phi=-\eps(-\partial_t\varphi)+\frac{n+1}{n+2}(-\psi+\Lambda)^{-\frac{1}{n+2}}(-\partial_t\psi)\\
&+(-\partial_t\varphi)\big(\eps(-\Delta_{\varphi}\varphi)+\frac{n+1}{n+2}(-\psi+\Lambda)^{-\frac{1}{n+2}}\Delta_{\varphi}\psi+\frac{n+1}{(n+2)^2}(-\psi+\Lambda)^{-\frac{n+3}{n+2}}|\nabla_{\varphi}\psi|^2\big)\\
&\ge -\eps(-\partial_t\varphi)+\frac{n+1}{n+2}(-\psi+\Lambda)^{-\frac{1}{n+2}}(-\partial_t\psi)+(-\partial_t\varphi)\big(\eps tr_{\varphi}g-\eps n\\
&+\frac{n+1}{n+2}(-\psi+\Lambda)^{-\frac{1}{n+2}}tr_{\omega_{\varphi}}\omega_{\psi}-\frac{n+1}{n+2}(-\psi+\Lambda)^{-\frac{1}{n+2}}tr_{\varphi}g\big).
\end{split}
\end{equation*}
Because of our choice of $\eps$ and $\Lambda$,  we have that 
\begin{equation*}
\eps \ge \frac{n+1}{n+2}\Lambda^{-\frac{1}{n+2}}\ge \frac{n+1}{n+2}(-\psi+\Lambda)^{-\frac{1}{n+2}}.
\end{equation*}
Moreover,  we use arithmetic-geometric inequality to obtain:
\begin{equation*}
tr_{\omega_{\varphi}}\omega_{\psi}\ge n\big(\frac{\omega_{\psi}^n}{\omega_{\varphi}^n}\big)^{\frac{1}{n}}=n\big(\frac{-\partial_t\varphi}{-\partial_t\psi}\frac{\eta_j(-\varphi-s)}{A_{s}}\big)^{\frac{1}{n}}.
\end{equation*}
Therefore
\begin{equation*}
\begin{split}
&L\Phi\ge -\eps(-\partial_t\varphi)+\frac{n+1}{n+2}(-\psi+\Lambda)^{-\frac{1}{n+2}}(-\partial_t\psi)\\
&+n\frac{n+1}{n+2}(-\psi+\Lambda)^{-\frac{1}{n+2}}\frac{\eta_j^{\frac{1}{n}}(-\varphi-s)}{A_s^{\frac{1}{n}}}\frac{(-\partial_t\varphi)^{1+\frac{1}{n}}}{(-\partial_t\psi)^{\frac{1}{n}}}.
\end{split}
\end{equation*}
Now we wish to use the following inequality:
\begin{equation*}
(BA^{\frac{1}{n}})^{\frac{n}{n+1}}x\le Ay+B\frac{x^{1+\frac{1}{n}}}{y^{\frac{1}{n}}}.
\end{equation*}
This essentially follows from Young's inequality,  by writing $x=\frac{x}{y^{\frac{1}{n+1}}}\cdot y^{\frac{1}{n+1}}$,  with exponents $n+1$,  $\frac{n+1}{n}$.
Now we use the above inequality with $y=-\partial_t\psi$,  $x=-\partial_t\varphi$,  $A=\frac{n+1}{n+2}(-\psi+\Lambda)^{-\frac{1}{n+2}}$,  $B=n\frac{n+1}{n+2}(-\psi+\Lambda)^{-\frac{1}{n+2}}\eta_j^{\frac{1}{n}}(-\varphi-s).$
Hence we get that
\begin{equation*}
\begin{split}
&\frac{n+1}{n+2}(-\psi+\Lambda)^{-\frac{1}{n+2}}(-\partial_t\psi)+n\frac{n+1}{n+2}(-\psi+\Lambda)^{-\frac{1}{n+2}}\frac{\eta_j^{\frac{1}{n}}(-\varphi-s)}{A_s^{\frac{1}{n}}}\frac{(-\partial_t\varphi)^{1+\frac{1}{n}}}{(-\partial_t\psi)^{\frac{1}{n}}}\\
&\ge n^{\frac{n}{n+1}}\frac{n+1}{n+2}(-\psi+\Lambda)^{-\frac{1}{n+2}}\frac{\eta_j^{\frac{1}{n+1}}(-\varphi-s)}{A_s^{\frac{1}{n+1}}}(-\partial_t\varphi).
\end{split}
\end{equation*}
Therefore we obtain that
\begin{equation}\label{2.1}
L\Phi\ge (-\partial_t\varphi)(-\eps+n^{\frac{n}{n+1}}\frac{n+1}{n+2}(-\psi+\Lambda)^{-\frac{1}{n+2}}\eta_j^{\frac{1}{n+1}}(-\varphi-s)A_s^{-\frac{1}{n+1}}).
\end{equation}
Assume that $\Phi$ achieves positive maximum at $(t_0,x_0)$.  Since we have chosen $s\ge ||\varphi_0||_{L^{\infty}}$,  we have $t_0>0$,  so that $L\Phi(t_0,x_0)\le 0$,  and $\Phi(t_0,x_0)>0$.  In particular,  at $(t_0,x_0)$,  we have
\begin{equation*}
\eps^{\frac{1}{n+1}}\eta_j^{\frac{1}{n+1}}(-\varphi-s)(-\psi+\Lambda)^{\frac{1}{n+2}}\ge \eps^{\frac{1}{n+1}}(-\varphi-s)^{\frac{1}{n+1}}(-\psi+\Lambda)^{-\frac{1}{n+2}}>1.
\end{equation*}
So that from (\ref{2.1}),  we get
\begin{equation*}
0\ge L\Phi(t_0,x_0)\ge (-\partial_t\varphi)(-\eps+n^{\frac{n}{n+1}}\frac{n+1}{n+2}\eps^{-\frac{1}{n+1}}A_s^{-\frac{1}{n+1}})>0.
\end{equation*}
The last $>0$ follows from our choice of $\eps$.  So that we get a contradiction and $\Phi\le 0$ on $[0,T]\times M$.  
\end{proof}

Using the above lemma,  the Proposition \ref{p2.3} follows immediately.  
\begin{proof}
Let $\alpha_0$ be as given by Corollary \ref{c2.2}.  Then we have:
\begin{equation*}
\int_M\exp\big(\alpha_0c_nA_{j,s}^{-\frac{1}{n+1}}((-\varphi-s)^+)^{\frac{n+2}{n+1}}\omega_0^n\le \int_Me^{-\alpha_0\psi_j}e^{C_nA_{j,s}}\omega_0^n.
\end{equation*}
The right hand side of the equation (\ref{2.1NN}) is given by
$\frac{\eta_j(-\varphi-s)e^F\omega_0^n}{A_{j,s}}$,  which has integral $=1$ from the definition of $A_{j,s}$.  Also the initial value of $\psi_j$ is zero.  Hence we are in a position to apply Corollary \ref{c2.2} to conclude that
\begin{equation*}
\sup_{t\in[0,T]}\int_Me^{-\alpha_0\psi_j}\omega_0^n\le C.
\end{equation*}
Here $C$ depends only on the background metric.
Now we pass to limit as $j\rightarrow \infty$.  It is easy to see that $A_{j,s}\rightarrow A_s$.
Take sup in $t\in[0,T]$,  the result immediately follows,  since $A_s\le E$.
\end{proof}
In our situation,  we have automatically have an upper bound for $E$ above. 
Indeed,  we use the inequality $xy\le x\log(x)+e^{y-1}$ for $x>0,\,y>0$,  we have
\begin{equation*}
\int_M(-\varphi)e^F\omega_0^n=\int_M(-\alpha_0\varphi)\frac{1}{\alpha_0}e^F\omega_0^n\le \int_Me^{-\alpha_0\varphi}\omega_0^n+\int_M\log(\frac{1}{\alpha_0}e^F)\frac{1}{\alpha_0}e^F\omega_0^n.
\end{equation*}
So that
\begin{equation*}
\int_{[0,T]\times M}(-\varphi)e^F\omega_0^ndt\le T\sup_{t\in[0,T]}\int_Me^{-\alpha_0\varphi}\omega_0^n+\int_{[0,T]\times M}\log(\frac{1}{\alpha_0}e^F)\frac{1}{\alpha_0}e^F\omega_0^n.
\end{equation*}
Therefore,  we have:
\begin{prop}\label{p2.5}
\begin{equation*}
\sup_{t\in[0,T]}\int_Me^{\beta_0 A_s^{-\frac{1}{n+2}}((-\varphi-s)^+)^{\frac{n+2}{n+1}}}\omega_0^n\le C.
\end{equation*}
Here $C$ depends on the background metric,  $T$,  an upper bound for $||\varphi_0||_{L^{\infty}}$ as well as entropy,  defined as
\begin{equation*}
Ent(F)=\int_{[0,T]\times M}(F^2+1)^{\frac{1}{2}}e^F\omega_0^ndt.
\end{equation*}
\end{prop}
Using the above estimate,  we would like to show that
\begin{lem}\label{l2.7New}
Let $\varphi$ be the solution to (\ref{2.1NNN}),  such that for some $p>n+1$,
\begin{equation*}
Ent_p(F):=\int_{[0,T]\times M}(F^2+1)^{\frac{p}{2}}e^F\omega_0^ndt<\infty.
\end{equation*}
Define $\phi(s)=\int_{\varphi<-s}e^F\omega_0^ndt$.  Then for some $B_0>0$,  and $\delta=\frac{1}{n+1}-\frac{1}{p}$,  we have
\begin{equation*}
r\phi(s+r)\le B_0\phi(s)^{1+\delta},
\end{equation*}
for any $r>0$,  $s\ge ||\varphi_0||_{L^{\infty}}$.
\end{lem}
\begin{proof}
First we note that:
\begin{equation}\label{2.4N}
\begin{split}
&(\frac{\beta_0}{2})^pA_s^{-\frac{p}{n+1}}((-\varphi-s)^+)^{p\frac{n+2}{n+1}}(e^F+1)\le (1+e^F)(1+\log(1+e^F))^p\\
&+C(p)e^{\beta_0A_s^{-\frac{1}{n+1}}((-\varphi-s)^+)^{\frac{n+2}{n+1}}}.
\end{split}
\end{equation}
In the above,  we used the following Lemma \ref{l2.7} with $x=\frac{\beta_0}{2}A_s^{-\frac{1}{n+1}}((-\varphi-s)^+)^{\frac{n+2}{n+1}}$,  $y=\log(e^F+1)$.
Integrate both sides of (\ref{2.4N}),  we obtain that
\begin{equation}\label{2.5N}
\int_{[0,T]\times M}((-\varphi-s)^+)^{p\frac{n+2}{n+1}}e^F\omega_0^ndt\le CA_s^{\frac{p}{n+1}}.
\end{equation}
The constant $C$ above depends on $T$ and an upper bound for $Ent_p(F)$.
On the other hand
\begin{equation*}
\begin{split}
&A_s=\int_{[0,T]\times M}(-\varphi-s)^+e^F\omega_0^ndt\le \big(\int_{[0,T]\times M}((-\varphi-s)^+)^{p\frac{n+2}{n+1}}e^F\omega_0^n\big)^{\frac{n+1}{p(n+2)}}\big(\int_{\varphi<-s}e^F\omega_0^n\big)^{1-\frac{n+1}{p(n+2)}}\\
&\le A_s^{\frac{1}{n+2}}C^{\frac{n+1}{p(n+2)}}\big(\int_{\varphi<-s}e^F\omega_0^ndt\big)^{1-\frac{n+1}{p(n+2)}}.
\end{split}
\end{equation*}
Here the $C$ is the same constant $C$ on the right hand side of (\ref{2.5N}).
Therefore
\begin{equation*}
A_s\le C^{\frac{1}{p}}\big(\int_{\varphi<-s}e^F\omega_0^ndt\big)^{\frac{n+2}{n+1}-\frac{1}{p}}.
\end{equation*}
On the other hand,  
\begin{equation*}
A_s\ge r\int_{\varphi<-s-r}e^F\omega_0^ndt.
\end{equation*}
Putting $B_0=C^{\frac{1}{p}}$,  the result follows.
\end{proof}
The following is the elementary lemma used in the above proof.
\begin{lem}\label{l2.7}
Let $x>0$,  $y>0$,  and $p>1$,  we then have:
\begin{equation*}
x^pe^y\le e^y(1+y)^p+C(p)e^{2x}.
\end{equation*}
\end{lem}
\begin{proof}
Define $h(y)=x^pe^y-e^y(1+y)^p.$ Then
\begin{equation*}
h'(y)=x^pe^y-e^y(1+y)^p-e^yp(1+y)^{p-1}.
\end{equation*}
Let $y_0$ be such that $h'(y_0)=0$,  in other words,  
\begin{equation*}
x^p=(1+y_0)^p+p(1+y_0)^{p-1}.
\end{equation*}
Therefore,  $h(y)\le h(y_0)=e^{y_0}(x^p-(1+y_0)^p)=pe^{y_0}(1+y_0)^{p-1}$.
Here we note that $h(y)\rightarrow -\infty$ as $y\rightarrow +\infty$ for any fixed $x$.  
On the other hand,  $x^p\ge (1+y_0)^p$,  which implies $y_0\le x-1$,  so that 
\begin{equation*}
x^pe^y-e^y(1+y)^p\le pe^{x-1}x^p\le C(p)e^{2x}.
\end{equation*}
\end{proof}
The boundedness of $\varphi$ follows from the following lemma,  which is first due to De Giorgi and was also used in \cite{EGZ3},  \cite{Kolo98}.
\begin{lem}\label{l2.8}
Let $\phi:\mathbb{R}^+\rightarrow \mathbb{R}^+$ be a monotone decreasing function such that for some $\delta>0$ and any $s\ge s_0$,  $r>0$, 
\begin{equation*}
r\phi(s+r)\le B_0\phi(s)^{1+\delta}.
\end{equation*}
Then $\phi(s)\equiv 0$ for $s\ge \frac{2B_0\phi^{\delta}(s_0)}{1-2^{-\delta}}+s_0$.
\end{lem}
\begin{proof}
We define $s_k$ inductively for $k\ge 1$ so that:
\begin{equation*}
s_{k+1}-s_k=2B_0\phi(s_k)^{\delta}.
\end{equation*}
Then we have
\begin{equation*}
\phi(s_{k+1})\le \frac{B_0\phi(s_k)^{1+\delta}}{s_{k+1}-s_k}\le \frac{1}{2}\phi(s_k).
\end{equation*}
So that $\phi(s_k)\le 2^{-k}\phi(s_0)$,  hence $s_{k+1}-s_k\le 2B_02^{-k\delta}\phi^{\delta}(||\varphi_0||_{L^{\infty}})$.  Summing up,  we get that $\sum_{k\ge 0}(s_{k+1}-s_k)\le \sum_{k\ge 0}2B_02^{-k\delta}\phi^{\delta}(s_0)\le \frac{2B_0\phi^{\delta}(s_0)}{1-2^{-\delta}}$.
\end{proof}

\subsection{Estimate the H\"older continuity}
We wish to prove a parabolic version of H\"older continuity of the solution when the right hand side is in $L^{p_0}([0,T]\times M)$ for some $p_0>1$,  namely Theorem \ref{t1.2}.  Throughout this section,  we denote $q_0=\frac{p_0}{p_0-1}$.

Similar to the elliptic case,  the proof of H\"older continuity relies on two ingredients: the first one being the stability result and the second one being the approximation result.  We will use PDE approach to establish the stability result,  and we need to use Demailly's technique to construct the approximation,  when the standard mollification trick no longer works on manifolds.

Let $\delta>0$ and $v$ be a smooth function on $[0,T]\times M$ such that $\partial_tv\le 0$ and for each $t$,  $\omega_0+\sqrt{-1}\partial\bar{\partial}v\ge -\frac{\delta}{2}\omega_0$,  we wish to get a weighted Moser-Trudinger inequality,  similar to the Proposition \ref{p2.3}.  

Let $0<\delta<1$ and $s>0$,  we consider the function $(1-\delta)v-\varphi-s$.  Let $\eta_j:\mathbb{R}\rightarrow \mathbb{R}_+$ be a sequence of smooth functions such that $\eta_j\rightarrow \max(x,0)$ pointwise.  
Define 
\begin{equation*}
A_{s,j,\delta}=\int_{[0,T]\times M}\eta_j((1-\delta)v-\varphi-s)e^F\omega_0^ndt.
\end{equation*}
We the consider the solution $\psi_{j,\delta}$ to the following parabolic equation:
\begin{equation*}
\begin{split}
&(-\partial_t\psi_{j,\delta})\omega_{\psi_{j,\delta}}^n=\frac{\eta_j((1-\delta)v-\varphi-s)}{A_{s,j,\delta}}e^F\omega_0^n,\\
&\psi_{j,\delta}(0,\cdot)=0.
\end{split}
\end{equation*}
\begin{lem}\label{l2.9}
Let $s>||((1-\delta)v-\varphi)^+||_{L^{\infty}}$,  then there exists dimensial constant $c_n$,  $C_n$,  such that on $[0,T]\times M$,
\begin{equation*}
c_nA_{j,s,\delta}^{-\frac{1}{n+2}}((1-\delta)v-\varphi-s)\le (-\psi_j+C_n\delta^{-(n+2)}A_{j,s,\delta})^{\frac{n+1}{n+2}}.
\end{equation*}
\end{lem}
\begin{proof}
It is very similar to the proof of Lemma \ref{l2.4},  so we will be brief at certain places.
Let $\eps=(n^{\frac{n}{n+1}}\frac{n+1}{n+2})^{\frac{n+1}{n+2}}A_{j,s,\delta}^{-\frac{1}{n+2}}$,  $\Lambda=(\frac{n+1}{n+2})^{n+2}\eps^{-(n+2)}2^{n+2}\delta^{-(n+2)}=\frac{n+1}{n+2}n^{-n}2^{n+2}\delta^{-(n+2)}A_{j,s,\delta}$,  and we define $\Phi=\eps((1-\delta)v-\varphi-s)-(-\psi+\Lambda)^{\frac{n+1}{n+2}}$.  Note that $\Delta_{\varphi}v\ge -(1+\frac{\delta}{2})tr_{\varphi}\omega_0$,  we can compute:
\begin{equation*}
\begin{split}
&\Delta_{\varphi}\Phi=\eps((1-\delta)\Delta_{\varphi}v-\Delta_{\varphi}\varphi)+\frac{n+1}{n+2}(-\psi+\Lambda)^{-\frac{1}{n+2}}+\frac{n+1}{(n+2)^2}(-\psi+\Lambda)^{-\frac{n+3}{n+2}}|\nabla_{\varphi}\psi|^2\\
&\ge \eps(-(1-\delta)(1+\frac{\delta}{2})tr_{\varphi}\omega_0-n+tr_{\varphi}\omega_0)+\frac{n+1}{n+2}(-\psi+\Lambda)^{-\frac{1}{n+2}}\Delta_{\varphi}\psi\\
&\ge -\eps n+(\eps \frac{\delta}{2}-\frac{n+1}{n+2}(-\psi+\Lambda)^{-\frac{1}{n+2}})tr_{\varphi}g+\frac{n+1}{n+2}(-\psi+\Lambda)^{-\frac{1}{n+2}}tr_{\omega_{\varphi}}\omega_{\psi}\\
&\ge -\eps n+n\big(\frac{-\partial_t\varphi}{-\partial_t\psi}\frac{\eta_j((1-\delta)v-\varphi-s)}{A_{j,s,\delta}}\big)^{\frac{1}{n}}.
\end{split}
\end{equation*}
In the above,  we used that,  according to our choice of $\eps$ and $\Lambda$,  we have:
\begin{equation*}
\eps \frac{\delta}{2}\ge \frac{n+1}{n+2}\Lambda^{-\frac{1}{n+2}}\ge \frac{n+1}{n+2}(-\psi+\Lambda)^{-\frac{1}{n+2}}.
\end{equation*}
On the other hand
\begin{equation*}
\begin{split}
&\partial_t\Phi=\eps((1-\delta)\partial_tv-\partial_t\varphi)-\frac{n+1}{n+2}(-\psi+\Lambda)^{-\frac{1}{n+2}}(-\partial_t\psi)\\
&\le \eps(-\partial_t\varphi)-\frac{n+1}{n+2}(-\psi+\Lambda)^{-\frac{1}{n+2}}(-\partial_t\psi).
\end{split}
\end{equation*}
Therefore
\begin{equation}
\begin{split}
&L\Phi=-\partial_t\Phi+(-\partial_t\varphi)\Delta_{\varphi}\Phi\ge -\eps(n+1)(-\partial_t\varphi)+\frac{n+1}{n+2}(-\psi+\Lambda)^{-\frac{1}{n+2}}(-\partial_t\psi)\\
&+\frac{n(n+1)}{n+2}A_{j,s,\delta}^{-\frac{1}{n}}(-\psi+\Lambda)^{-\frac{1}{n+2}}(((1-\delta)v-\varphi-s)^+)^{\frac{1}{n}}\frac{(-\partial_t\varphi)^{1+\frac{1}{n}}}{(-\partial_t\psi)^{\frac{1}{n}}}\\
&\ge -\eps(n+1)(-\partial_t\varphi)+n^{\frac{n}{n+1}}\frac{n+1}{n+2}A_{j,s,\delta}^{-\frac{1}{n+1}}(-\psi+\Lambda)^{-\frac{1}{n+2}}\eta_j^{\frac{1}{n+1}}((1-\delta)v-\varphi-s)(-\partial_t\varphi).
\end{split}
\end{equation}
Note that according to our choice of $\eps$,  we have that
\begin{equation*}
-\eps+n^{\frac{n}{n+1}}\frac{n+1}{n+2}A_{j,s,\delta}^{-\frac{1}{n+1}}\eps^{-\frac{1}{n+1}}>0.
\end{equation*}
Therefore $L\Phi>0$.
Since we have chosen $s>||\eta_j((1-\delta)v-\varphi-s)||_{L^{\infty}}$ for large enough $j$,  we have that $\Phi\le 0$ for $t=0$.  Hence $\Phi\le 0$ for all $(t,x)\in[0,T]\times M$ by maximum principle.
\end{proof}
As a direct consequence of Lemma \ref{l2.9},  we have 
\begin{cor}\label{c2.10}
There exist constants $\beta_0>0$,  $C>0$,  depending only on dimension and background metric,  such that for any $s\ge ||((1-\delta)v_0-\varphi_0-s)^+||_{L^{\infty}}$ and any $0<\delta<1$,
\begin{equation*}
\sup_{t\in[0,T]}\int_M\exp\big(\beta_0A_{s,\delta}^{-\frac{1}{n+1}}(((1-\delta)v-\varphi-s)^+)^{\frac{n+2}{n+1}}\big)\omega_0^n\le \exp\big(C\delta^{-(n+2)}A_{s,\delta}\big).
\end{equation*}
Here 
\begin{equation*}
A_{s,\delta}=\int_{[0,T]\times M}((1-\delta)v-\varphi-s)^+e^F\omega_0^ndt.
\end{equation*}
\end{cor}
Using the above weighted Moser-Trudinger inequality,  we would like to estimate $\sup_M((1-\delta)v-\varphi-s)$.  Before we can do that,  we wish to estimate $(1-\delta)v-\varphi-s$ in $L^p$ for any $p<\infty$.
For this we have the following lemma:
\begin{lem}\label{l2.11}
Let $\varphi$ solve (\ref{2.1NNN}) with the right hand side $e^F\in L^{p_0}([0,T]\times M,\omega_0^ndt)$.  Let $0<\delta<1$,  and we assume that $v$ is a function on $[0,T]\times M$ such that $\partial_tv\le 0$,  $\omega_0+\sqrt{-1}\partial\bar{\partial}v\ge -\frac{\delta}{2}\omega_0$.  Let $s_0$ be chosen so that:
\begin{enumerate}
\item $s_0\ge ||((1-\delta)v_0-\varphi_0)^+||_{L^{\infty}}$,
\item $A_{s_0,\delta}\le \delta^{n+2}$.
\end{enumerate}
Then for any $p>1$,  we have
\begin{equation*}
||((1-\delta)v-\varphi-s)^+||_{L^p(\omega_0^ndt)}\le C(p)A_{s,\delta}^{\frac{1}{n+2}}.
\end{equation*}
Here $C(p)$ depends on dimension,  the background metric,  $p$ and $T$.
\end{lem}
\begin{proof}
According to Corollary \ref{c2.10},  if $s_0$ satisfies the assumptions (1) and (2) above,  we would have:
\begin{equation*}
\sup_{t\in[0,T]}\int_M\exp\big(\beta_0A_{s,\delta}^{-\frac{1}{n+1}}(((1-\delta)v-\varphi-s)^+)^{\frac{n+2}{n+1}}\big)\omega_0^n\le e^C.
\end{equation*}
Here the $C$ depends only on dimension and background metric.  Therefore,  for any positive integer $p$,  we have that:
\begin{equation*}
\sup_{t\in[0,T]}\int_M\frac{1}{p!}\beta_0^pA_{s,\delta}^{-\frac{p}{n+1}}\big(((1-\delta)v-\varphi-s)^+\big)^{\frac{n+2}{n+1}p}\omega_0\le e^C.
\end{equation*}
Therefore
\begin{equation*}
\int_{[0,T]\times M}(((1-\delta)v-\varphi-s)^+)^{\frac{n+2}{n+1}p}\omega_0^ndt\le e^CTp!\beta_0^pA_{s,\delta}^{\frac{p}{n+1}}.
\end{equation*}
If we put $p'=\frac{n+2}{n+1}p$,  the result then follows.
\end{proof}
Using this,  we can conclude that
\begin{lem}\label{l2.12}
Under the same assumption as Lemma \ref{l2.11},  we have that for any $\eta<\frac{1}{q_0(n+1)}$,  there exists a constant $B_0$,  depending only on $\eta$,  $T$,  $||e^F||_{L^{p_0}}$,  the background metric,  and the dimension.
\begin{equation*}
rvol(\Omega_{s+r,\delta})\le B_0vol^{1+\eta}(\Omega_s).
\end{equation*}
\end{lem}
\begin{proof}
Using H\"older inequality,  we have
\begin{equation*}
\begin{split}
&A_{s,\delta}=\int_{[0,T]\times M}((1-\delta)v-\varphi-s)^+e^F\omega_0^ndt\le ||((1-\delta)v-\varphi-s)^+||_{L^{q_0}(\omega_0^ndt)}||e^F||_{L^{p_0}(\omega_0^ndt)}\\
&\le ||e^F||_{L^{p_0}}||((1-\delta)v-\varphi-s)^+||_{L^{q_0\beta}}vol^{\frac{1}{q_0}(1-\frac{1}{\beta})}(\Omega_{s,\delta}).
\end{split}
\end{equation*}
Using Lemma \ref{l2.11} we get that
\begin{equation*}
||((1-\delta)v-\varphi-s)^+||_{L^{q_0\beta}}\le C(q_0,\beta)A_{s,\delta}^{\frac{1}{n+2}}.
\end{equation*}
So that
So that
\begin{equation*}
A_{s,\delta}^{\frac{n+1}{n+2}}\le ||e^F||_{p_0}C(q_0,\beta)vol^{\frac{1}{q_0}(1-\frac{1}{\beta})}(\Omega_s).
\end{equation*}
On the other hand
\begin{equation*}
\begin{split}
&||((1-\delta)v-\varphi-s)^+||_{L^1}\le vol^{1-\frac{1}{\beta}}(\Omega_s)||((1-\delta)v-\varphi-s)^+||_{L^{\beta}}\le vol^{1-\frac{1}{\beta}}(\Omega_s)C_1(\beta)A_{s,\delta}^{\frac{1}{n+2}}\\
&\le C_1(\beta)C^{\frac{1}{n+1}}(q_0,\beta)||e^F||_{L^{p_0}}^{\frac{1}{n+1}}vol^{(1+\frac{1}{q_0(n+1)})(1-\frac{1}{\beta})}(\Omega_{s,\delta}).
\end{split}
\end{equation*}
On the other hand,  
\begin{equation*}
||((1-\delta)v-\varphi-s)^+||_{L^1}\ge rvol(\Omega_{s+r,\delta}).
\end{equation*}
Therefore,  
\begin{equation*}
rvol(\Omega_{s+r,\delta})\le C_1(\beta)C^{\frac{1}{n+1}}(q_0,\beta)||e^F||_{L^{p_0}}^{\frac{1}{n+1}}vol^{(1+\frac{1}{q_0(n+1)})(1-\frac{1}{\beta})}(\Omega_{s,\delta}).
\end{equation*}
\end{proof}
Using Lemma \ref{l2.8},  we can conclude the following proposition.

\begin{prop}\label{p2.13}
Let $\varphi$ solve (\ref{2.1NNN}) with the right hand side $e^F\in L^{p_0}([0,T]\times M,\omega_0^ndt)$.  Assume that $v$ is a function on $[0,T]\times M$ such that $\partial_tv\le 0$,  $\omega_0+\sqrt{-1}\partial\bar{\partial}v\ge -\frac{\delta}{2}\omega_0$.  Let $s_0$ be chosen so that:
\begin{enumerate}
\item $s_0\ge ||((1-\delta)v_0-\varphi_0)^+||_{L^{\infty}}$,
\item $A_{s_0,\delta}\le \delta^{n+2}$.
\end{enumerate}
Then for any $\mu<\frac{1}{q_0(n+1)}$,  we have
\begin{equation*}
\sup((1-\delta)v-\varphi)\le s_0+C\big(vol(\{(1-\delta)v-\varphi-s_0>0\})\big)^{\mu}.
\end{equation*}
Here $vol(E)=\int_E\omega_0^ndt$ for any measurable subset of $[0,T]\times M$.
Here $C$ depends on the choice of $\mu<\frac{1}{q_0(n+1)}$,  an upper bound for $||e^F||_{L^{p_0}}$,  $T$,  the background metric,  and dimension.
\end{prop}
\begin{proof}
We wish to apply Lemma \ref{l2.8} with $\phi(s)=vol(\Omega_{s,\delta}).$ Then we have that $\phi(s)\equiv 0$ for $s\ge \frac{2B_0\phi^{\eta}(s_0)}{1-2^{-\eta}}+s_0$.  This precisely means that:
\begin{equation*}
\sup((1-\delta)v-\varphi)\le \frac{2B_0\phi^{\eta}(s_0)}{1-2^{-\eta}}+s_0.
\end{equation*}
Note that according to Lemma \ref{l2.12},  $\eta$ can be taken to be any value less than $\frac{1}{q_0(n+1)}$.
\end{proof}
We further note that if $v$ is bounded:
\begin{equation*}
vol(\Omega_{s_0,\delta})\le \frac{1}{s_0}\int_{\Omega_{s_0,\delta}}((1-\delta)v-\varphi)^+\omega_0^ndt\le \frac{1}{s_0}\big(||(v-\varphi)^+||_{L^1}+\delta||v||_{L^{\infty}}vol(\Omega_{s_0,\delta})\big).
\end{equation*}
Therefore,  if we choose $s_0$ so that $s_0\ge 2\delta||v||_{L^{\infty}}$,  we would have that 
\begin{equation*}
vol(\Omega_{s_0,\delta})\le \frac{2}{s_0}||(v-\varphi)^+||_{L^1}.
\end{equation*}
Therefore,  we have the following consequence of Proposition \ref{p2.13}:
\begin{cor}
Assume that $v\le 0$ and $v$ is bounded,  with $\omega_0+\sqrt{-1}\partial\bar{\partial}v\ge -\frac{\delta}{2}\omega_0$.  Assume also that $s_0$ is chosen so that 
\begin{enumerate}
\item $s_0\ge||((1-\delta)v_0-\varphi_0)^+||_{L^{\infty}}$,
\item $A_{s_0,\delta}\le \delta^{n+2}$.
\item $s_0\ge 2\delta||v||_{L^{\infty}}$.
\end{enumerate}
Then for any $\mu<\frac{1}{q_0(n+1)}$,  we have
\begin{equation*}
\sup_{[0,T]\times M}(v-\varphi)\le s_0+C_3s_0^{-\mu}||(v-\varphi)^+||_{L^1}^{\mu}.
\end{equation*}
\end{cor}

Given $0<\delta<1$,  the set of $s_0$ satisfying (1)-(3) above are given by $[s_*(\delta),+\infty)$,  where $s_*(\delta)$ is the infimum of $s_0$ satisfying the above 3 conditions.
Therefore,  from the above corollary,  we get that 
\begin{equation}\label{2.7NNN}
\sup_{[0,T]\times M}(v-\varphi)\le \inf_{0<\delta<1}\inf_{s_0\ge s_*(\delta)}\big(s_0+C_3s_0^{-\mu}||(v-\varphi)^+||_{L^1}^{\mu}\big).
\end{equation}
Next we will estimate $s_*(\delta)$ for any given $0<\delta<1$,  then we will choose a suitable $\delta$.
We make the following claim:
\begin{lem}\label{l2.15}
Given $0<\delta<1$,  and given $\beta>1$,  there exists a constant $C_1(\beta)$,  depending on the choice of $\beta$,  such that
\begin{equation}
s_*(\delta)\le \max\big(2||(v_0-\varphi_0)^+||_{L^{\infty}},2\delta||v||_{L^{\infty}},C_1(\beta)\delta^{-\frac{q_0(n+1)}{1-\frac{1}{\beta}}}||(v-\varphi)^+||_{L^1}\big).
\end{equation}
\end{lem}
\begin{proof}
Clearly we have that $s_*(\delta)\ge\max(||((1-\delta)v_0-\varphi_0)^+||_{L^{\infty}},2\delta||v||_{L^{\infty}})$.
From now on,  we will denote $s_*(\delta)$ simply as $s_*$,  for the simplicity of notations. 
If we have $>$ holds in the above,  then we would have $A_{s_*,\delta}=\delta^{n+2}$. 
On the other hand,
\begin{equation*}
\begin{split}
&A_{s_*,\delta}=\int_{\Omega_{s_*,\delta}}((1-\delta)v-\varphi-s_*)^+e^F\omega_0^ndt\le ||e^F||_{L^{p_0}}||((1-\delta)v-\varphi-s_*)^+||_{L^{q_0}}\\
&\le ||e^F||_{L^{p_0}}||((1-\delta)v-\varphi-s_*)^+||_{L^{\beta q_0}}vol^{\frac{1}{q_0}(1-\frac{1}{\beta})}(\Omega_{s_*,\delta})\\
&\le ||e^F||_{L^{p_0}}C(\beta,q_0)e^{C_0\delta^{-(n+2)}A_{s_*,\delta}}A_{s,\delta}^{\frac{1}{n+2}}vol^{\frac{1}{q_0}(1-\frac{1}{\beta})}(\Omega_{s_*,\delta})\\
&\le ||e^F||_{L^{p_0}}C(\beta,q_0)e^{C_0\delta^{-(n+2)}A_{s_*,\delta}}A_{s_*,\delta}^{\frac{1}{n+2}}(\frac{2}{s_*}||(v-\varphi)^+||_{L^1})^{\frac{1}{q_0}(1-\frac{1}{\beta})}.
\end{split}
\end{equation*}
Since we have that $A_{s_*,\delta}=\delta^{n+2}$,  we get that
\begin{equation*}
\delta^{n+1}\le ||e^F||_{p_0}C(\beta,q_0)(\frac{2}{s_*}||(v-\varphi)^+||_{L^1})^{\frac{1}{q_0}(1-\frac{1}{\beta})}.
\end{equation*}
In other words,  we get
\begin{equation*}
s_*\le C_1(\beta)\delta^{-\frac{q_0(n+1)}{1-\frac{1}{\beta}}}||(v-\varphi)^+||_{L^1}.
\end{equation*}
Therefore,  we have that
\begin{equation*}
\begin{split}
&s_*(\delta)\le \max\big(||((1-\delta)v_0-\varphi_0)^+||_{L^{\infty}},2\delta||v||_{L^{\infty}},C_1(\beta)\delta^{-\frac{q_0(n+1)}{1-\frac{1}{\beta}}}||(v-\varphi)^+||_{L^1}\big)\\
&\le \max\big(||(v_0-\varphi_0)^+||_{L^{\infty}}+\delta||v_0||_{L^{\infty}},2\delta||v||_{L^{\infty}},C_1(\beta)\delta^{-\frac{q_0(n+1)}{1-\frac{1}{\beta}}}||(v-\varphi)^+||_{L^1}\big)\\
&\le \max\big(2||(v_0-\varphi_0)^+||_{L^{\infty}},2\delta||v||_{L^{\infty}},C_1(\beta)\delta^{-\frac{q_0(n+1)}{1-\frac{1}{\beta}}}||(v-\varphi)^+||_{L^1}\big).
\end{split}
\end{equation*}
\end{proof}

Next we will combine (\ref{2.7NNN}) and Lemma \ref{l2.15} to estimate $\sup(v-\varphi)$.  That is,  we will minimize $\inf_{0<\delta<1}\inf_{s_0\ge s_*(\delta)}(s_0+C_3s_0^{-\mu}||(v-\varphi)^+||_{L^1}^{\mu})$ subject to the constraint given by Lemma \ref{l2.15}.
We eventually obtain the following parabolic analogue of Theorem 1. 1 in \cite{GKZ}:
\begin{thm}\label{t2.2}
Let $\varphi$ solves (\ref{2.1NNN}) with initial value $\varphi_0$ such that the right hand side $e^F\in L^{p_0}(\omega_0^ndt)$ for some $p_0>1$.  Let $v$ be a bounded function defined on $[0,T]\times M$ such that $\partial_tv\le 0$ and $\omega_0+\sqrt{-1}\partial\bar{\partial}v\ge 0$.  Then for any $\alpha<\frac{1}{1+q_0(n+1)}$,
\begin{equation*}
\sup_{[0,T]\times M}(v-\varphi)\le C\max\big(||(v_0-\varphi_0)^+||_{L^{\infty}},||(v-\varphi)^+||_{L^1}^{\alpha}\big).
\end{equation*}
In the above,  $v_0=v|_{t=0}$ and $C$ depends only on the background metric,  $||e^F||_{L^{p_0}}$,  $T$,  $n$ and choice of $\alpha<\frac{1}{1+q_0(n+1)}$.
\end{thm}
\begin{proof}
Since we know that both $v$ and $\varphi$ are bounded in $L^{\infty}$ with the said dependence,  we may assume that $||(v-\varphi)^+||_{L^1}<1$,  without loss of generality.
Take $\delta=||(v-\varphi)^+||_{L^1}^{\frac{1}{1+q_0(n+1)}}$.  We take $\beta>1$ large enough such that $1-\frac{q_0(n+1)}{1+q_0(n+1)}\frac{1}{1-\frac{1}{\beta}}>\alpha$.  Let $C_1(\beta)$ be the constant given by Lemma \ref{l2.15} with this $\beta$.  Now we define $C_4=\max(2||v||_{L^{\infty}},C_1(\beta))$.  We wish to define 
\begin{equation}\label{2.9}
s_0=\max(2||v_0-\varphi_0||_{L^{\infty}},C_4||(v-\varphi)^+||_{L^1}^{\alpha}).
\end{equation}
We would like to use (\ref{2.7NNN}) to estimate $\sup_M(v-\varphi)$ by taking $\delta$ and $s_0$ as specified above,  where we take $\mu$ so that $\alpha=\frac{\mu}{1+\mu}$.
Of  course we need to verify that $s_0\ge s_*(\delta)$.  

Once this is verified and $s_0=2||v_0-\varphi_0||_{L^{\infty}}$, in (\ref{2.9}),  namely $2||v_0-\varphi_0||_{L^{\infty}}\ge C_4||(v-\varphi)^+||_{L^1}^{\alpha}$,  we get from (\ref{2.7NNN})
\begin{equation*}
\begin{split}
&\sup_{[0,T]\times M}(v-\varphi)\le 2||v_0-\varphi_0||_{L^{\infty}}+C_3(C_4||(v-\varphi)^+||_{L^1}^{\alpha})^{-\mu}||(v-\varphi)^+||_{L^1}^{\mu}\\
&\le 2||v_0-\varphi_0||_{L^{\infty}}+C_3C_4^{-\mu}||(v-\varphi)^+||_{L^1}^{\alpha}\le (2+2C_3C_4^{-1-\mu})||v_0-\varphi_0||_{L^{\infty}}.
\end{split}
\end{equation*}
If $s_0=C_4||(v-\varphi)^+||_{L^1}^{\alpha}$,  then from (\ref{2.7NNN}),  
\begin{equation*}
\begin{split}
&\sup_M(v-\varphi)\le C_4||(v-\varphi)^+||_{L^1}^{\alpha}+C_3(C_4||(v-\varphi)^+||_{L^1}^{\alpha})^{-\mu}||(v-\varphi)^+||_{L^1}^{\mu}\\
&\le (C_4+C_3C_4^{-\mu})||(v-\varphi)^+||_{L^1}^{\alpha}.
\end{split}
\end{equation*}
So we just need to verify that $s_0\ge s_*(\delta)$,  for which Lemma \ref{l2.15} will be needed.  It is clear that we only need to show that:
\begin{equation*}
C_4||(v-\varphi)^+||_{L^1}^{\alpha}\ge \max\big(2\delta||v||_{L^{\infty}},C_1(\beta)\delta^{-\frac{q_0(n+1)}{1-\frac{1}{\beta}}}||(v-\varphi)^+||_{L^1}\big).
\end{equation*}
This is clear from our choice of $\delta$,  $C_4$ and $\beta$ made at the beginning of the proof.
\end{proof}
Using this theorem,  we can prove H\"older continuity in time.  Indeed,  we may define
\begin{equation}
\varphi_{1,\eps}(t,x)=\frac{1}{\eps}\int_{t-\eps}^t\varphi(\tau,x)d\tau.
\end{equation}
Here we have extended $\varphi(t,x)=\varphi_0(x)$ for $t<0$.  Then we have that $\partial_t\varphi_{1,\eps}\le 0$ and $\omega_0+\sqrt{-1}\partial\bar{\partial}\varphi_{1,\eps}\ge 0$.
Moreover,  because of the extension we chose,  we have that $\varphi_{1,\eps}(0,x)=\varphi_0$.  Therefore,  by taking $v=\varphi_{1,\eps}$,  we get that for any $\alpha<\frac{1}{1+q_0(n+1)}$:
\begin{equation}
\sup_{[0,T]\times M}(\varphi_{1,\eps}-\varphi)\le C||(\varphi_{1,\eps}-\varphi)^+||_{L^1}^{\alpha}.
\end{equation}
But we note that $\varphi_{1,\eps}\ge \varphi$,  we get
\begin{equation*}
\begin{split}
&||\varphi_{1,\eps}-\varphi||_{L^1}=\frac{1}{\eps}\int_0^Tdt\int_{t-\eps}^t\int_M\varphi(\tau,x)d\tau\omega_0^n-\int_{[0,T]\times M}\varphi(t,x)\omega_0^ndt\\
&=\frac{1}{\eps}\int_{-\eps}^0\int_M(\tau+\eps)\varphi(\tau,x)\omega_0^nd\tau+\frac{1}{\eps}\int_{T-\eps}^T\int_M(T-\tau)\varphi(\tau,x)\omega_0^nd\tau\\
&-\int_{T-\eps}^T\int_M\varphi(t,x)\omega_0^ndt.
\end{split}
\end{equation*}
Now it is clear that 
\begin{equation*}
||\varphi_{1,\eps}-\varphi||_{L^1}\le C\eps.
\end{equation*}
Therefore,  Theorem \ref{t2.2} implies that for any $\alpha<\frac{1}{1+q_0(n+1)}$,  we have
\begin{equation*}
\frac{1}{\eps}\int_{t-\eps}^t\varphi(\tau,x)d\tau-\varphi(t,x)\le C\eps^{\alpha},\,\,(t,x)\in[0,T]\times M.
\end{equation*}
The H\"older continuity in time is implied by the following lemma:
\begin{lem}
Let $f(t)$ be a decreasing function on $[0,T]$.  Assume that there is $C_0>0$ such that for any $\eps>0$ and any $t\in[0,T]$,  it holds:
\begin{equation*}
\frac{1}{\eps}\int_{t-\eps}^tf(\tau)d\tau-f(t)\le C_0\eps^{\alpha}.
\end{equation*}
Then we have
\begin{equation*}
\sup_{s\neq t}\frac{|f(t)-f(s)|}{|t-s|^{\alpha}}\le C.
\end{equation*}
\end{lem}
\begin{proof}
We may calculate
\begin{equation*}
\begin{split}
&C_0\eps^{\alpha}\ge \frac{1}{\eps}\int_{t-\eps}^tf(\tau)d\tau-f(t)=\frac{1}{\eps}\int_{t-\eps}^td\tau\int_{\tau}^t(-f'(s))ds\ge \frac{1}{\eps}\int_{t-\eps}^{t-\frac{\eps}{2}}d\tau\int_{\tau}^t(-f'(s))ds\\
&\ge \frac{1}{\eps}\int_{t-\eps}^{t-\frac{\eps}{2}}d\tau\int_{t-\frac{\eps}{2}}^t(-f'(s))ds=\frac{1}{2}\int_{t-\frac{\eps}{2}}^t(-f'(s))ds.
\end{split}
\end{equation*}
Therefore,  for any $t>s$,  we take $\eps=2(t-s)$,  the claimed estimate follows.
\end{proof}
Hence we obtain:
\begin{prop}
Let $\varphi$ solves (\ref{2.1NNN}) with the right hand side $e^F\in L^{p_0}(\omega_0^ndt)$ for some $p_0>1$,  then for any $\alpha<\frac{1}{1+q_0(n+1)}$,
\begin{equation*}
|\varphi(t,x)-\varphi(s,x)|\le C|t-s|^{\alpha}.
\end{equation*}
Here the constant $C$ depends only on $||e^F||_{L^{p_0}}$,  the background metric,  the dimension,  the choice of $\alpha<\frac{1}{1+q_0(n+1)}$,  and $T$.
\end{prop}

In order to estimate the regularity of $\varphi$ in space,  we need Demailly's approximation technique in \cite{DDGHKZ}.  This construction is a substitute for the standard mollification of a function in Euclidean spaces.  Indeed,  if $u$ is a psh function defined on a domain $\Omega\subset \mathbb{C}^n$,  then on a suitable subdomain of $\Omega$,  we may define $u_{\eps}(z)=\int_{B_{\eps}(0)}u(z-w)\frac{1}{\eps^{2n}}\rho(\frac{|w|}{\eps})dV(\omega)$.  Here $dV(\omega)$ means the standard volume form on $\mathbb{C}^n$ and $\rho$ is a smoothing kernel such that $\int_{\mathbb{C}^n}\rho(|w|)dV(\omega)=1$.
We would automatically have that $u_{\eps}$ is also psh and $u_{\eps}\rightarrow u$.

However,  things are getting subtle on a manifold and instead,  one needs to consider:
\begin{equation}\label{2.10New}
\rho_{\eps}u(x)=\int_{T_xM}u(\exp_x( \zeta))\frac{1}{\eps^{2n}}\rho(\frac{|\zeta|^2_{\omega_0}}{\eps^2})dV_{\omega_0}(\zeta),\,\,\,x\in M.
\end{equation}
Here $u$ is a psh function on $M$ and $\rho:\mathbb{R}\rightarrow \mathbb{R}^+$ is supported on $[0,1]$ such that $\int_{\mathbb{C}^n}\rho(|z|)=1$. 
However,  if we do this,  we lack good control over the lower bound of $\omega_0+\sqrt{-1}\partial\bar{\partial}\rho_{\eps}u$.  
In order to have good estimate on the lower bound of the complex Hessian,  we need to consider the Kiselman-Legendre transform,  as was done in \cite{BD} and \cite{Demailly92}:
\begin{equation*}
U_{c,\eps}(x)=\inf_{0<s\le \eps}\big(\rho_su(x)+Ks^2-K\eps^2-c\log(\frac{s}{\eps})\big).
\end{equation*}
Here $c>0$,  $K>0$.  We need the following facts:
\begin{prop}\label{p2.16}
\begin{enumerate}
\item (\cite{DDGHKZ},  Lemma 2.1)For $K$ large enough depending only on the background metric,   $s\mapsto \rho_su(x)+Ks^2$ is convex and increasing such that $\lim_{s\rightarrow 0^+}(\rho_su(x)+Ks^2)=u(x)$.
\item (\cite{DDGHKZ},  Lemma 2.1) Assume that $u$ is $\omega_0$-psh,  then there is a constant $A$ large enough depending only on the background metric,  such that
\begin{equation*}
\omega_0+\sqrt{-1}\partial\bar{\partial}U_{c,\eps}\ge -A(c+\eps^2)\omega_0.
\end{equation*}
\item (\cite{DDGHKZ},  Lemma 2.3)There is a constant $C>0$ depending only on the background metric and $||u||_{L^{\infty}}$ such that $\int_M|\rho_{\eps}u-u|\omega_0^n\le C\eps^2$
\item (\cite{BD},  in the proof of Lemma 1.12 and in particular,  equation (1.16)) For some $c>0$,  $C>0$,  $\eps_0>0$ depending only on the background metric,  it holds that: for any $z\in M$,  if we choose normal coordinates at $z$ (meaning $(\omega_0)_{i\bar{j}}(z)=\delta_{ij}$,  $\nabla g_{i\bar{j}}(z)=0$),  then for any $\eps<\eps_0$
\begin{equation*}
\rho_{\eps}u(z)-u(z)\ge \frac{c}{\eps^{2n-2}}\int_{B(z,\frac{\eps}{2})}\Delta u(z')dV(z')-C\eps^2.
\end{equation*}
Here $\Delta u=\sum_{i=1}^n\frac{\partial^2u}{\partial z_i\partial\bar{z}_i}$ and $B(z,\frac{\eps}{2})$ is a ball under the normal coordinates.
\end{enumerate}
\end{prop}
\begin{rem}\label{r2.17}
The point (4) of the above Proposition can be reformulated as:

There exists $\eps_0>0$,  such that for $0<\eps<\eps_0$,  and any coordinate chart $U$,  we have that
\begin{equation*}
\rho_{\eps}u-u\ge \frac{c}{\eps^{2n-2}}\int_{B(z,\frac{\eps}{2})}\Delta u(z')dV(z')-C\eps^2
\end{equation*}
holds for any $z\in V\subset \subset U$.
Here $\Delta=\sum_{i=1}^n\frac{\partial^2u}{\partial z_i\partial\bar{z}_i}$.  Indeed,  by choosing a local potential $\rho_0$ with $\omega_0=\sqrt{-1}\partial\bar{\partial}\rho_0$ such that $\sqrt{-1}\partial\bar{\partial}(u+\rho_0)\ge 0$,  the $\Delta(u+\rho_0)$ operator calculated under different coordinates will bound each other by a positive multiple.  So point (4) will still hold,  possibly with a different choice of the constants (but still only depend on the background manifold and metric).
\end{rem}
Let $\varphi(t,x)$ be the solution to (\ref{2.1NNN}),  we wish to apply the above approximation to $\varphi(t,x)$ for each $t\in[0,T]$.

Indeed,  let $\gamma=\bar{\alpha}$ if $\bar{\alpha}<\frac{1}{1+q_0(n+1)}$ and $\gamma<\frac{1}{1+q_0(n+1)}$ if $\bar{\alpha}\le \frac{1}{1+q_0(n+1)}$,  we define
\begin{equation}\label{2.13NNew}
\varphi_{\eps}(t,x)=\inf_{0<s\le \eps}\big(\rho_{s}\varphi(t,x)+Ks^2-K\eps^2-\eps^{\gamma}\log(\frac{s}{\eps})\big).
\end{equation}
Then from Proposition \ref{p2.16},  we have that for $0<\eps<1$ and $A>0$ depending only on the background metric:
\begin{equation*}
\omega_0+\sqrt{-1}\partial\bar{\partial}\tilde{\varphi}_{\eps}\ge -A(\eps^{\gamma}+\eps^2)\omega_0\ge -2A\eps^{\gamma}\omega_0.
\end{equation*}

Next we will choose $\delta=4A\eps^{\gamma}$,  then we have $\omega_0+\sqrt{-1}\partial\bar{\partial}\varphi_{\eps}\ge -\frac{\delta}{2}\omega_0$,  hence it follows from (\ref{2.7NNN}) that for any $\mu<\frac{1}{1+q_0(n+1)}$:
\begin{equation}\label{2.14NNew}
\sup_{[0,T]\times M}(\varphi_{\eps}-\varphi)\le \inf_{s_0\ge s_*(\delta)}\big(s_0+C_3s_0^{-\mu}||(\varphi_{\eps}-\varphi)^+||_{L^1}^{\mu}\big).
\end{equation}
We have the following estimate on $||(\varphi_{\eps}-\varphi)^+||_{L^1}$:
\begin{lem}
Let $\varphi_{\eps}$ be defined by (\ref{2.13NNew}),  then we have
\begin{equation*}
||(\varphi_{\eps}-\varphi)^+||_{L^1}\le C\eps^2.
\end{equation*}
Here $C$ depends on the background metric and $T$.
\end{lem}
\begin{proof}
We note that,  by taking $s=\eps$ in (\ref{2.13NNew}),  we have that
\begin{equation*}
\varphi_{\eps}(t,x)\le \rho_{\eps}\varphi(t,x).
\end{equation*}
Therefore,  if we use point (3) of Proposition \ref{p2.16},  we get that
\begin{equation*}
\int_M|\rho_{\eps}\varphi-\varphi|(t,x)\omega_0^n\le C\eps^2.
\end{equation*}
Here $C$ depends only on the background metric.  Hence if we integrate in $t$,  we get that
\begin{equation*}
\int_{[0,T]\times M}|\rho_{\eps}-\rho|\omega_0^ndt\le CT\eps^2.
\end{equation*}
\end{proof}
Hence (\ref{2.14NNew}) implies:
\begin{equation}\label{2.15NNN}
\sup_{[0,T]\times M}(\varphi_{\eps}-\varphi)\le \inf_{s_0\ge s_*(\delta)}\big(s_0+C_4s_0^{-\mu}\eps^{2\mu}\big).
\end{equation}
We need to plug in a suitable value of $s_0$,  hence we need to estimate $s_*(\delta)$ using Lemma \ref{l2.15}.
With our choice of $\delta=4A\eps^{\gamma}$,  we get from Lemma \ref{l2.15} that:
\begin{equation}\label{2.15New}
s_*(\delta)\le \max\big(2||(\varphi_{\eps,0}-\varphi_0)^+||_{L^{\infty}},8A\eps^{\gamma}||\varphi_{\eps}||_{L^{\infty}},C_2(\beta)\eps^{2-\frac{2\gamma q_0(n+1)}{1-\frac{1}{\beta}}}\big).
\end{equation}
Since $\varphi_0\in C^{\bar{\alpha}}$,  we get that
\begin{equation}\label{2.16New}
\varphi_{\eps,0}-\varphi_0\le \rho_{\eps}\varphi_0-\varphi_0\le C\eps^{\bar{\alpha}}\le C\eps^{\gamma}.
\end{equation}
Moreover,  since $\varphi\in L^{\infty}$,  we would get
\begin{equation}\label{2.17New}
||\varphi_{\eps}||_{L^{\infty}}\le ||\varphi||_{L^{\infty}}+K\eps^2\le ||\varphi||_{L^{\infty}}+1.
\end{equation}
Finally,  since $\gamma<\frac{1}{1+q_0(n+1)}$,  we see that,  if we choose $\beta>1$ large enough,  we may secure that $2-\frac{2\gamma q_0(n+1)}{1-\frac{1}{\beta}}>\gamma$.  Hence if we combine (\ref{2.15New})-(\ref{2.17New}),  we get that
\begin{equation}
s_*(\delta)\le C_5\eps^{\gamma}.
\end{equation}
Here $C_5$ depends on $||\varphi||_{L^{\infty}}$,  $T$,  the background metric,  the $C^{\bar{\alpha}}$ norm of $\varphi_0$.

Hence if we use (\ref{2.15NNN}),  and we take $s_*(\delta)=C_5\eps^{\gamma}$,  we obtain that:
\begin{equation}
\sup_{[0,T]\times M}(\varphi_{\eps}-\varphi)\le C_5\eps^{\gamma}+C_4(C_5\eps^{\gamma})^{-\mu}\eps^{2\mu}.
\end{equation}
Note that $\gamma<\frac{2}{1+q_0(n+1)}$,  we see that if we take $\mu$ sufficiently close to $\frac{1}{1+q_0(n+1)}$,  one can make $2\mu-\gamma \mu>\gamma$.
Therefore,  we get 
\begin{lem}\label{l2.21}
Define $\varphi_{\eps}$ according to (\ref{2.13NNew}).  Then we have,  for any choice of $\gamma$ such that $\gamma\le \bar{\alpha}$ and $\gamma<\frac{1}{1+q_0(n+1)}$,  we have:
\begin{equation*}
\sup_{[0,T]\times M}(\varphi_{\eps}-\varphi)\le C\eps^{\gamma}.
\end{equation*}
Here $C$ depends on the choice of $\gamma<\frac{1}{1+q_0(n+1)}$,  the $C^{\bar{\alpha}}$ norm of $\varphi_0$,  $T$ and the background metric.
\end{lem}
Using (\ref{2.13NNew}),  we get that:
\begin{lem}
There exists $0<\theta<1$,  such that
\begin{equation*}
\rho_{\theta\eps}\varphi(t,x)-\varphi(t,x)\le C\eps^{\gamma}.
\end{equation*}
Here $\theta$,  $C$ have the same dependence as in Lemma \ref{l2.21}.
\end{lem}
\begin{proof}
Let us go back to (\ref{2.13NNew}) and we consider two cases: $0<s\le \theta \eps$ and $\theta \eps\le s\le 1$,  where $0<\theta<1$ to be chosen.

First,  if $0<s\le \theta \eps$,  we see that,  since $s\mapsto \rho_s\varphi(t,x)+Ks^2$ is monotone increasing and tends to $\varphi(t,x)$ as $s\rightarrow 0$,  we obtain that
\begin{equation*}
\rho_s\varphi+Ks^2-K\eps^2-\eps^{\gamma}\log(\frac{s}{\eps})\ge \varphi(t,x)-K\eps^2+\eps^{\gamma}\log(\frac{1}{\theta}).
\end{equation*}

If $\theta\eps\le s\le \eps$,  then we have:
\begin{equation*}
\rho_s\varphi+Ks^2-K\eps^2-\eps^{\gamma}\log(\frac{s}{\eps})\ge \rho_{\theta\eps}\varphi+K(\theta \eps)^2-K\eps^2\ge \rho_{\theta\eps}\varphi-K\eps^2.
\end{equation*}
That is,  from Lemma \ref{l2.21}:
\begin{equation*}
C\eps^{\gamma}\ge \varphi_{\eps}-\varphi\ge \min\big(-K\eps^2+\eps^{\gamma}\log(\frac{1}{\theta}),\rho_{\theta\eps}\varphi-\varphi-K\eps^2\big).
\end{equation*}
Now we wish to choose $\theta>0$ small enough so that $\log(\frac{1}{\theta})-K>C$,  so that 
\begin{equation*}
-K\eps^2+\eps^{\gamma}\log(\frac{1}{\theta})\ge \eps^{\gamma}(\log(\frac{1}{\theta})-K)>C\eps^{\gamma}.
\end{equation*}
Hence with this choice of $\theta$ we get
\begin{equation*}
\rho_{\theta\eps}\varphi-\varphi\le K\eps^2+C\eps^{\gamma}\le (C+K)\eps^{\gamma}.
\end{equation*}
\end{proof}
Let $U$ be an open neighborhood of $M$ which is bi-holomorphic to a domain in $\mathbb{C}^n$.  Denote the local coordinate on $U$ to be $(z_1,z_2,\cdots,z_n)$.  Let $V\subset \subset U$,  then from point (4) of Proposition \ref{p2.16} and Remark \ref{r2.17},  we see that,  there is $\eps_0>0$,  such that for any $0<\eps<\eps_0$ and any $z\in V$,  we would have:
\begin{equation*}
C\eps^{\gamma}\ge \rho_{\theta\eps}\varphi-\varphi\ge \frac{c}{(\theta\eps)^{2n-2}}\int_{B(z,\frac{\eps}{2})}\Delta\varphi(t,z')dV(z')-C\eps^2.
\end{equation*}
Then the desired H\"older continuity follows from the following elementary result:
\begin{lem}
Let $u$ be a bounded function defined on $\Omega\subset \mathbb{R}^n$ and $\Omega'\subset \subset \Omega$.  Assume that there exists $r_0>0$,  $C_0>0$,  $0<\alpha<1$,  such that for any $z\in \Omega'$,  any $0<r<r_0$,  we have
\begin{equation*}
\int_{B(z,r)}|\Delta u|(z')dV(z')\le C_0r^{n-2+\alpha}.
\end{equation*}
Then 
\begin{equation*}
\sup_{x,y\in\Omega'}|u(x)-u(y)|\le C|x-y|^{\alpha}.
\end{equation*}
\end{lem}
Here the constant $C$ depends on dimension,  $r_0$,  $dist(\Omega',\partial\Omega)$,  $C_0$,  and $||u||_{L^{\infty}}$.
\begin{proof}
Choose $\Omega_1$ such that $\Omega'\subset\subset
\Omega_1\subset\subset\Omega$.  Denote $f=\Delta u$,  we define 
\begin{equation*}
\bar{u}(x)=\int_{\Omega_1}f(y)G(x,y)dV(y).
\end{equation*}
Here $G(x,y)$ is the fundamental solution to $\Delta$ in $\mathbb{R}^n$.  Let us assume that $n\ge 3$ so that we have $G(x,y)=c_n|x-y|^{2-n}$.
Then $\Delta(u-\bar{u})=0$ on $\Omega_1$ and $\Delta\bar{u}=f$.  We just need to do H\"older estimate for $\bar{u}$.  Choose $x_1,\,x_2\in \Omega'$ such that $r=|x_1-x_2|<\frac{r_0}{10}$.  Denote $x_*=\frac{1}{2}(x_1+x_2)$,  we can then compute
\begin{equation}\label{2.21NNew}
\begin{split}
&\bar{u}(x_1)-\bar{u}(x_2)=\int_{B(x_*,3r)}f(y)G(x_1,y)dV(y)-\int_{B(x_*,3r)}f(y)G(x_2,y)dV(y)\\
&+\int_{\Omega_1-B(x_*,3r)}f(y)(G(x_1,y)-G(x_2,y))dV(y)
\end{split}
\end{equation}
For the first term above,
\begin{equation}\label{2.22NNew}
\begin{split}
&|\int_{B(x_*,3r)}f(y)G(x_1,y)dV(y)|\le \int_{B(x_1,4r)}|f(y)|c_n|x_1-y|^{2-n}dV(y)\\
&=\int_0^{4r}c_n\rho^{2-n}d\rho\int_{\partial B(x_1,\rho)}|f(y)|d\mathcal{H}^{n-1}(y)=\int_0^{4r}c_n\rho^{2-n}\frac{d}{d\rho}\int_{B(x_1,\rho)}|f(y)|dV(y)\\
&=c_n(4r)^{2-n}\int_{B(x_1,4r)}|f(y)|dV(y)+\int_0^{4r}c_n(n-2)\rho^{1-n}\int_{B(x_1,\rho)}|f(y)|dV(y)d\rho\\
&\le c_n(4r)^{2-n}C_0(4r)^{n-2+\alpha}+\int_0^{4r}c_n(n-2)\rho^{1-n}C_0\rho^{n-2+\alpha}\le Cr^{\alpha}.
\end{split}
\end{equation}
The estimate for the second term is the same.  We now look at the last term.  First we note that for $y\notin B(x_*,3r)$,  it holds:
\begin{equation*}
\begin{split}
&|G(x_1,y)-G(x_2,y)|\le |x_1-x_2|\sup_{0\le t\le 1}|\nabla_xG((1-t)x_1+tx_2,y)|\\
&\le C_n|x_1-x_2|\sup_{0\le t\le 1}|(1-t)x_1+tx_2-y|^{1-n}\le C_n'|x_1-x_2||x_*-y|^{1-n}.
\end{split}
\end{equation*}
Therefore,
\begin{equation}\label{2.23NNew}
\begin{split}
&|\int_{\Omega_1-B(x_*,3r)}f(y)(G(x_1,y)-G(x_2,y))dV(y)|\le \int_{\Omega_1-B(x_*,3r)}|f(y)|C_n'r|x_*-y|^{1-n}dV(y)\\
&\le C_n'r\big(\int_{B(x_*,r_0)-B(x_*,3r)}|f(y)||x_*-y|^{1-n}dV(y)+\int_{\Omega_1-B(x_*,r_0)}|f(y)||x_*-y|^{1-n}dV(y)\big).
\end{split}
\end{equation}
In the first term above,  we have:
\begin{equation}\label{2.24NNew}
\begin{split}
&\int_{B(x_*,r_0)-B(x_*,3r)}|f(y)||x_*-y|^{1-n}dV(y)=\int_{3r}^{r_0}\rho^{1-n}d\rho\int_{\partial B(x_*,\rho)}|f(y)|d\mathcal{H}^{n-1}(y)\\
&=\int_{3r}^{r_0}\rho^{1-n}d\rho\frac{d}{d\rho}\int_{B(x_*,\rho)}|f(y)|dV(y)\le r_0^{1-n}\int_{B(x_*,r_0)}|f(y)|dV(y)\\
&+\int_{3r}^{r_0}(n-1)\rho^{-n}\int_{B(x_*,\rho)}|f(y)|dV(y)d\rho\le r_0^{1-n}\int_{B(x_*,r_0)}|f(y)|dV(y)\\
&+\int_{3r}^{r_0}(n-1)\rho^{-n}C_0\rho^{n-2+\alpha}d\rho\le C(r_0)+Cr^{-1+\alpha}.
\end{split}
\end{equation}
For the second term in (\ref{2.23NNew}),  one has
\begin{equation}\label{2.25NNNew}
\int_{\Omega_1-B(x_*,r_0)}|f(y)||x_*-y|^{1-n}dV(y)\le r_0^{1-n}\int_{\Omega_1-B(x_*,r_0)}|f(y)|dV(y).
\end{equation}
Combining (\ref{2.21NNew})-(\ref{2.25NNNew}),  we would get that $|u(x_1)-u(x_2)|\le Cr^{\alpha}$.
\end{proof}

\section{$L^{\infty}$ estimate for more general parabolic Hessian equations}

In this section,  we would like to generalize the above $L^{\infty}$ estimates to more general Hessian equations:
\begin{equation}\label{3.1}
\begin{split}
&f(-\partial_t\varphi,\lambda[h_{\varphi}])=e^F,\\
&\varphi|_{t=0}=\varphi_0.
\end{split}
\end{equation}

$f(\lambda_0,\lambda_1,\cdots,\lambda_n)$ is a $C^1$ function defined on a cone $\Gamma\subset \{(\lambda_0,\cdots,\lambda_n):\sum_{i=0}^n\lambda_i>0\}$.  Here $(h_{\varphi})_j^i=g^{i\bar{k}}(g_{\varphi})_{j\bar{k}}$.  
The eigenvalues of $(h_{\varphi})_j^i$ does not depend on the choice of local coordinates,  which we will denote as $\lambda[h_{\varphi}]$.

We also assume that $\Gamma$ contains the positive cone $\Gamma_+$ given by $\{(\lambda_0,\cdots,\lambda_n):\lambda_0>0,\cdots,\lambda_n>0\}$.
Moreover,  we assume the following conditions,  similar to the conditions assumed in \cite{Phong}:
\begin{enumerate}
\item $\frac{\partial f}{\partial \lambda_i}>0$ for any $0\le i\le n$,
\item $f$ is a symmetric function in terms of $\lambda_1$,  $\cdots$,  $\lambda_n$.
\item For some $c_0>0$,  we have $\partial_0f\det(\frac{\partial f}{\partial h_{ij}})\ge c_0$ on the positive cone $\Gamma_+$.
\item For some $C_0>0$ such that $\sum_{i=0}^n\lambda_i\frac{\partial f}{\partial \lambda_i}\le C_0f$ on the positive cone $\Gamma_+$.
\end{enumerate}
\begin{rem}
The first condition above just guarantees that our equation is parabolic.  The second condition is a usual assumption for Hessian equation.   Note that condition 4 can be guaranteed as long as $f$ is a homogeneous function in terms of its variables.  The above considered parabolic complex Monge-Ampere equation is only a special case by taking $f(\lambda)=\big(\Pi_{i=0}^n\lambda_i\big)^{\frac{1}{n+1}}$,  where $\Gamma=\Gamma_+=\{\lambda_0>0,\,\cdots,\lambda_n>0\}$.
We can give a few examples which satisfies the conditions (1)-(4) above:
\begin{example}
Let $f(\lambda_1,\cdots,\lambda_n)$ be a positive function defined on $\Gamma$ which is contained in $\{(\lambda_1,\cdots,\lambda_n):\sum_i\lambda_i>0\}$ and contains $\{(\lambda_1,\cdots,\lambda_n):\lambda_i>0,\,1\le i\le n\}$.  Assume also that $f$ satisfies the conditions assumed in \cite{Phong},  namely $\frac{\partial f}{\partial \lambda_i}>0$ for $1\le i\le n$,  $f$ symmetric in $\lambda_1,\cdots,\lambda_n$,  $\det(\frac{\partial f}{\partial h_{ij}})\ge c_0$ and  $\sum_{i=1}^n\lambda_i\frac{\partial f}{\partial \lambda_i}\le C_0f$ on the positive cone.

Then the function $g(\lambda_0,\cdots,\lambda_n)=\big(\lambda_0f(\lambda_1,\cdots,\lambda_n)\big)^{\frac{n}{n+1}}$ satisfies the conditions (1)-(4) above,  as a function defined on $\{\lambda_0>0\}\times \Gamma$.  In particular,  $g(\lambda_0,\cdots,\lambda_n)=(\lambda_0\sigma_k^{\frac{1}{k}}(\lambda_1,\cdots,\lambda_n))^{\frac{n}{n+1}}$ for $1\le k\le n$ and $g(\lambda_0,\cdots,\lambda_n)=\big(\lambda_0(\frac{\sigma_k(\lambda_1,\cdots,\lambda_n)}{\sigma_l(\lambda_1,\cdots,\lambda_n)})^{\frac{1}{k-l}}\big)^{\frac{n}{n+1}}$ for $1\le l<k\le n$ satisfy conditions (1)-(4).
\end{example}
\begin{example}
For any $1\le k\le n+1$,  define $f(\lambda)=(\sigma_k(\lambda_0,\cdots,\lambda_n))^{\frac{1}{k}}$,  or $f(\lambda_0,\cdots,\lambda_n)=(\frac{\sigma_k}{\sigma_l}(\lambda_0,\cdots,\lambda_n))^{\frac{1}{k-l}}$ for $n\ge k>l\ge 1$ will satisfy the above conditions.
\end{example}

\end{rem}

In order to establish the $L^{\infty}$ bound,  first we wish to establish the analogue of Proposition \ref{p2.3}:
\begin{prop}
Let $\varphi$ be a solution to (\ref{3.1}) such that $\varphi_0$ is bounded.  Denote $A_s=\int_{[0,T]\times M}(-\varphi-s)^+e^F\omega_0^ndt$.  Then there exists constants $\beta_0>0$,  $C>0$,  depending only on the background metric,  the structural constants $c_0$ and $C_0$,  such that for any $s\ge ||\varphi_0||_{L^{\infty}}$,
\begin{equation*}
\sup_{t\in[0,T]}\int_Me^{\beta_0A_s^{-\frac{1}{n+2}}((-\varphi-s)^+)^{\frac{n+2}{n+1}}}\omega_0^n\le C\exp(CE).
\end{equation*}
Here $E=\int_{[0,T]\times M}(-\varphi)e^{(n+1)F}\omega_0^ndt.$
\end{prop}
Similar to Proposition \ref{p2.3},  we let $\psi_j$ be defined as:
\begin{equation*}
\begin{split}
&(-\partial_t\psi_j)\omega_{\psi_j}^n=\frac{\eta_j(-\varphi-s)e^{(n+1)F}\omega_0^n}{A_{j,s}},\\
&\psi_j|_{t=0}=0
\end{split}
\end{equation*}
\begin{lem}
There exists constants $c$ and $C$,  depending only on dimension and the structural constant $c_0$,  $C_0$ from $f$,   such that
\begin{equation*}
c_nA_{j,s}^{-\frac{1}{n+1}}((-\varphi-s)^+)^{\frac{n+2}{n+1}}\le -\psi_j+C_nA_{j,s}.
\end{equation*}
\end{lem}
\begin{proof}
The calculation follows similar lines as Lemma \ref{l2.4}.
We define the operator $L$ to be:
\begin{equation*}
Lu=-\partial_0f\partial_tu+\frac{\partial f}{\partial h_{ij}}g^{i\bar{k}}\partial_{j\bar{k}}u.
\end{equation*}
The following calculations are done pointwisely at any $z\in M$,  and we assume that the local coordinates have been chosen so that $g_{i\bar{j}}=\delta_{ij}$,  so that under this coordinate,  one has:
\begin{equation*}
Lu=-\partial_0f\partial_tu+\frac{\partial f}{\partial h_{ij}}\partial_{j\bar{i}}u.
\end{equation*}
We choose $\eps=(\frac{n+1}{n+2})^{\frac{n+1}{n+2}}c_0^{\frac{1}{n+2}}C_0^{-\frac{n+1}{n+2}}$,  $\Lambda=\eps^{-(n+2)}(\frac{1}{2}\frac{n+2}{n+1})^{n+2}$ and we similarly define
\begin{equation*}
\Phi=\eps(-\varphi-s)-(-\psi+\Lambda)^{\frac{n+1}{n+2}}.
\end{equation*}
Then we can compute 
\begin{equation}\label{3.2NN}
\partial_t\Phi=\eps(-\partial_t\varphi)-\frac{n+1}{n+2}(-\psi+\Lambda)^{-\frac{1}{n+2}}(-\partial_t\psi).
\end{equation}
Also
\begin{equation}\label{3.3NN}
\Phi_{i\bar{j}}=-\eps\varphi_{i\bar{j}}-\frac{n+1}{n+2}(-\psi+\Lambda)^{-\frac{1}{n+2}}(-\psi_{i\bar{j}})+\frac{n+1}{(n+2)^2}(-\psi+\Lambda)^{-\frac{n+3}{n+2}}\psi_i\psi_{\bar{j}}.
\end{equation}
Therefore,  
\begin{equation*}
\begin{split}
&\sum_{i,j}\frac{\partial f}{\partial h_{ij}}\Phi_{i\bar{j}}=-\eps\frac{\partial f}{\partial h_{ij}}\varphi_{i\bar{j}}+\frac{n+1}{n+2} (-\psi+\Lambda)^{-\frac{1}{n+2}}\frac{\partial f}{\partial h_{ij}}\psi_{i\bar{j}}+\frac{n+1}{(n+2)^2}(-\psi+\Lambda)^{-\frac{n+3}{n+2}}\frac{\partial f}{\partial h_{ij}}\psi_i\psi_{\bar{j}}\\
&\ge -\eps\frac{\partial f}{\partial h_{ij}}(\omega_{\varphi})_{i\bar{j}}+\eps\frac{\partial f}{\partial h_{ij}}g_{i\bar{j}}+\frac{n+1}{n+2}(-\psi+\Lambda)^{-\frac{1}{n+2}}\frac{\partial f}{\partial h_{ij}}(\omega_{\psi})_{i\bar{j}}-\frac{n+1}{n+2}(-\psi+\Lambda)^{-\frac{1}{n+2}}\frac{\partial f}{\partial h_{ij}}g_{i\bar{j}}
\end{split}
\end{equation*}

Combining,  we get
\begin{equation*}
\begin{split}
&L\Phi\ge -\eps(\partial_t\varphi\partial_0f+\frac{\partial f}{\partial h_{ij}}(\omega_{\varphi})_{i\bar{j}})+\frac{n+1}{n+2}(-\psi+\Lambda)^{-\frac{1}{n+2}}\partial_0f(-\partial_t\psi)\\
&+\frac{n+1}{n+2}(-\psi+\Lambda)^{-\frac{1}{n+2}}\frac{\partial f}{\partial h_{ij}}(\omega_{\psi})_{i\bar{j}}+(\eps-\frac{n+1}{n+2}(-\psi+\Lambda)^{-\frac{1}{n+2}})\frac{\partial f}{\partial h_{ij}}g_{i\bar{j}}
\end{split}
\end{equation*}
According to our choice of constants,  we have
\begin{equation*}
\eps \ge \frac{n+1}{n+2}\Lambda^{-\frac{1}{n+2}}.
\end{equation*}
So that
\begin{equation*}
\begin{split}
&L\Phi\ge -\eps\big(\partial_t\varphi\partial_0f+\frac{\partial f}{\partial h_{ij}}(\omega_{\varphi})_{i\bar{j}}\big)+\frac{n+1}{n+2}(-\psi+\Lambda)^{-\frac{1}{n+2}}\partial_0f(-\partial_t\psi)\\
&+\frac{n+1}{n+2}(-\psi+\Lambda)^{-\frac{1}{n+2}}\frac{\partial f}{\partial h_{ij}}(\omega_{\psi})_{i\bar{j}}.
\end{split}
\end{equation*}
Now we assume that $\Phi$ achieves positive maximum at $(t_0,x_0)$.  As before,  $\Phi\le 0$ at $t_0=0$ because of our choice of $s$,  so that we must have $t_0>0$.  Then at $(t_0,x_0)$,  it holds:
\begin{equation*}
\partial_t\Phi(t_0,x_0)\ge 0,\,\,\,\Phi_{i\bar{j}}(t_0,x_0)\le 0.
\end{equation*}
It follows from (\ref{3.2NN}) and (\ref{3.3NN}),  we have
\begin{equation*}
-\partial_t\varphi(t_0,x_0)>0,\,\,(\omega_0)_{i\bar{j}}+\varphi_{i\bar{j}}(t_0,x_0)>0.
\end{equation*} 
In particular,  this means that $(-\partial_t\varphi(t_0,x_0),(\omega_{\varphi})_{i\bar{j}}(t_0,x_0))\in\Gamma_+$.
Here we assumed to have chosen holomorphic normal coordinates at $x_0$ so that $(\omega_0)_{i\bar{j}}(x_0)=\delta_{ij}$.  At $(t_0,x_0)$,  we have 
\begin{equation*}
\begin{split}
&\partial_t\varphi\partial_0f+\frac{\partial f}{\partial h_{ij}}(\omega_{\varphi})_{i\bar{j}}\le C_0f,\\
&\partial_0f\det(\frac{\partial f}{\partial h_{ij}})\ge c_0.
\end{split}
\end{equation*}

This would imply that
\begin{equation*}
\begin{split}
&L\Phi\ge -\eps C_0f+\frac{n+1}{n+2}(-\psi+\Lambda)^{-\frac{1}{n+2}}\partial_0f(-\partial_t\psi)+\frac{n+1}{n+2}(-\psi+\Lambda)^{-\frac{1}{n+2}}\frac{\partial f}{\partial h_{ij}}(\omega_{\psi})_{i\bar{j}}\\
&\ge -\eps C_0f+\frac{n+1}{n+2}(-\psi+\Lambda)^{-\frac{1}{n+2}}\partial_0f(-\partial_t\psi)+n\frac{n+1}{n+2}(-\psi+\Lambda)^{-\frac{1}{n+2}}\big(\det(\frac{\partial f}{\partial h_{ij}})\det(\omega_{\psi})_{i\bar{j}})\big)^{\frac{1}{n}}\\
&\ge -\eps C_0f+\frac{n+1}{n+2}(-\psi+\Lambda)^{-\frac{1}{n+2}}\partial_0f(-\partial_t\psi)+n\frac{n+1}{n+2}(-\psi+\Lambda)^{-\frac{1}{n+2}}(\frac{c_0}{\partial_0f}\frac{(-\varphi-s)^+e^{(n+1)F}A_s^{-1}}{-\partial_t\psi})^{\frac{1}{n}}\\
&=-\eps C_0f+\frac{n+1}{n+2}(-\psi+\Lambda)^{-\frac{1}{n+2}}\partial_0f(-\partial_t\psi)\\
&+\frac{n(n+1)}{n+2}c_0^{\frac{1}{n}}A_s^{-\frac{1}{n}}(-\psi+\Lambda)^{-\frac{1}{n+2}}((-\varphi-s)^+)^{\frac{1}{n}}\frac{f^{\frac{n+1}{n}}}{(\partial_0f)^{\frac{1}{n}}(-\partial_t\psi)^{\frac{1}{n}}}.
\end{split}
\end{equation*}
Therefore,  we get
\begin{equation*}
\begin{split}
&L\Phi\ge -\eps C_0f+\frac{n+1}{n+2}(-\psi-\Lambda)^{-\frac{1}{n+2}}\partial_0f(-\partial_t\psi) \\
&+\frac{n(n+1)}{n+2}c_0^{\frac{1}{n}}A_s^{-\frac{1}{n}}(-\psi+\Lambda)^{-\frac{1}{n+2}}((-\varphi-s)^+)^{\frac{1}{n}}(-\partial_t\psi)^{-\frac{1}{n}}f^{\frac{n+1}{n}}(\partial_0f)^{-\frac{1}{n}}.
\end{split}
\end{equation*}
Now we use the inequality that
\begin{equation*}
Ay^n+By^{-1}\ge n^{-\frac{n}{n+1}}A^{\frac{1}{n+1}}B^{\frac{n}{n+1}}.
\end{equation*}
Here $y=(-\partial_t\psi)^{\frac{1}{n}}$,  $A=\frac{n+1}{n+2}(-\psi+\Lambda)^{-\frac{1}{n+2}}\partial_0f$,  $B=\frac{n(n+1)}{n+2}c_0^{\frac{1}{n}}A_s^{-\frac{1}{n}}(-\psi+\Lambda)^{-\frac{1}{n+2}}((-\varphi-s)^+)^{\frac{1}{n}}f^{\frac{n+1}{n}}(\partial_0f)^{-\frac{1}{n}}$. 
Therefore,  we get,  at $(t_0,x_0)$.  
\begin{equation*}
L\Phi\ge -\eps C_0f+\frac{n+1}{n+2}c_0^{\frac{1}{n+1}}(-\psi+\Lambda)^{-\frac{1}{n+2}}((-\varphi-s)^+)^{\frac{1}{n+1}}A_s^{-\frac{1}{n+1}}f.
\end{equation*}
Note that since $\Phi(t_0,x_0)>0$,  we have that
\begin{equation*}
((-\varphi-s)^+)^{\frac{1}{n+1}}(-\psi+\Lambda)^{-\frac{1}{n+2}}>\eps^{-\frac{1}{n+1}}.
\end{equation*}
Therefore,  at $(t_0,x_0)$,  we have
\begin{equation*}
0\ge L\Phi\ge -\eps C_0f+\frac{n+1}{n+2}c_0^{\frac{1}{n+1}}\eps^{-\frac{1}{n+1}}A_s^{-\frac{1}{n+1}}f>0.
\end{equation*}
The last $>0$ follows from our choice of $\eps$
and we get a contradiction and we must have $\Phi\le 0.$
\end{proof} 
As a direct consequence,  we get that
\begin{prop}\label{p3.4}
Denote $A_s=\int_{[0,T]\times M}(-\varphi-s)^+e^{(n+1)F}\omega_0^ndt$.  Then there exist constants $\beta_0>0$,  depending only on the background metric,  the structural constants $c_0$,  $C_0$ and $n$,  such that for any $s\ge ||\varphi_0||_{L^{\infty}}$
\begin{equation*}
\sup_{t\in[0,T]}\int_Me^{\beta_0 A_s^{-\frac{1}{n+2}}((-\varphi-s)^+)^{\frac{n+2}{n+1}}}\omega_0^n\le C\exp(CE).
\end{equation*}
Here 
\begin{equation*}
E=\int_{[0,T]\times M}(-\varphi)e^{(n+1)F}\omega_0^n.
\end{equation*}
\end{prop}
We still lack a uniform bound for the term $E$.  In the case of complex Monge-Ampere,  we did it quite easily,  using the exponential integral bound for the $\omega_0$-psh function.  However,  in the case of complex Hessian equations,  all we have is that $(-\partial_t\varphi,\lambda[h_{\varphi}])\in\Gamma$,  and we no longer have an exponential bound for the solution $\varphi$.  So we have to work harder to achieve this bound.  To be more precise,  we have the following estimates:
\begin{prop}
Let $\varphi$ solves (\ref{3.1}) such that the right hand side satisfies $Ent_p(F):=\int_{[0,T]\times M}e^{(n+1)F}(|F|^p+1)\omega_0^ndt<\infty$ for some $p>0$.  Then we have:
\begin{enumerate}
\item If $p<n+1$,  then there exists a constant $\beta_0$,  which depends only on the structural constants $c_0$,  $C_0$,  background metric,  and upper bound for $Ent_p(F)$,  such that
\begin{equation*}
\sup_{[0,T]}\int_M\exp\big(\beta_0((-\varphi)^+)^{\frac{n+1}{n+1-p}}\big)\omega_0^n\le C.
\end{equation*}
Here $C$ has the same dependence as $\eps$,  but additionally on $T$ and $||\varphi_0||_{L^{\infty}}$.
\item If $p\ge n+1$,  then for any $N>0$,  there exists a constant $\beta_0$,  which depends only on the structural constants $c_0$,  $C_0$,  the background metric,  and upper bound for $Ent_p(F)$,  such that
\begin{equation*}
\sup_{t\in[0,T]}\int_M\exp\big(\beta_0((-\varphi)^+)^N\big)\omega_0^n\le C.
\end{equation*}
Here $C$ has the same dependence as $\eps$,  but additionally on $T$,  $||\varphi_0||_{L^{\infty}}$ and the choice of $N$.
\end{enumerate}
\end{prop}
The proof of this proposition follows similar ideas as the proof for the Moser-Trudinger inequality above (Proposition \ref{p2.5} and \ref{p3.4}).

Let $\psi$ be the solution to the following problem:
\begin{equation}\label{3.4}
\begin{split}
&(-\partial_t\psi)\omega_{\psi}^n=\frac{e^{(n+1)F}(F^2+1)^{\frac{p}{2}}\omega_0^n}{\int_{[0,T]\times M}e^{(n+1)F}(F^2+1)^{\frac{p}{2}}\omega_0^ndt},\\
&\psi|_{t=0}=0.
\end{split}
\end{equation}
Note that the integral of the right hand side of (\ref{3.4}) is 1,  and the initial value of $\psi$ is 0,  Corollary \ref{c2.2} implies that 
\begin{equation}\label{3.5NNew}
\sup_{t\in[0,T]}\int_Me^{-\alpha_0\psi}\omega_0^n\le C_*.
\end{equation}
Here both $\alpha_0$ and $C$ depends only on the background metric.

We have the following lemma holds:
\begin{lem}
Let $\varphi$ solve (\ref{3.1}) and $s\ge ||\varphi_0||_{L^{\infty}}$ 
We have:
\begin{enumerate}
\item If $p<n+1$,  then there exist constants $\eps$,  $\Lambda$,  depending only on the structural constants $c_0$,  $C_0$,  background metric,  an upper bound for $Ent_p(F)$,  such that 
\begin{equation*}
\eps(-\varphi-s)\le (-\psi+\Lambda)^{\frac{n+1}{n+1-p}}+C.
\end{equation*}
Here $C$ has the same dependence as $\eps$,  $\Lambda$,  but additionally on $T$.
\item If $p\ge n+1$,  then for any $N>0$,  there exists constants $\eps$,  $\Lambda$,  depending on the structural constants $c_0$,  $C_0$,  the background metric,  an upper bound for $Ent_p(F)$,  and $N$ such that
\begin{equation*}
\eps(-\varphi-s)\le (-\psi+\Lambda)^N+C.
\end{equation*}
Here $C$ has the same dependence as $\eps$,  $\Lambda$,  but additionally on $T$.
\end{enumerate}
\end{lem}
\begin{proof}
For $p<n+1$,  we choose $\alpha=1-\frac{p}{n+1}$,  and if $p\ge n+1$,  we choose $\alpha=\frac{1}{N}$.
We define
\begin{equation*}
\rho=-\eps t+\eps(-\varphi-s)-(-\psi+\Lambda)^{\alpha}.
\end{equation*}
In the above
\begin{equation*}
\eps=\frac{1}{C_0}(n+1)\alpha\big(\frac{\alpha_0}{n+1}\big)^{\frac{p}{n+1}}\big(c_0Ent_p^{-1}(F)\big)^{\frac{1}{n+1}},\,\,\,\Lambda=(\frac{4\alpha}{\eps})^{\frac{1}{1-\alpha}}.
\end{equation*}
Then we may compute
\begin{equation*}
\begin{split}
&L\rho=\partial_0f(-\rho_t)+\frac{\partial f}{\partial h_{ij}}g^{i\bar{k}}\rho_{j\bar{k}}=-\partial_0f(-\eps-\eps\partial_t\varphi-\alpha(-\psi+\Lambda)^{\alpha-1}(-\partial_t\psi))\\
&+\frac{\partial f}{\partial h_{ij}}g^{i\bar{k}}\big(-\eps\varphi_{j\bar{k}}+\alpha(-\psi+\Lambda)^{\alpha-1}\psi_{j\bar{k}}+\alpha(1-\alpha)(-\psi+\Lambda)^{\alpha-2}\psi_j\psi_{\bar{k}}\big)\\
&\ge -\eps \partial_0f(-\partial_t\varphi)+\eps\partial_0f+\alpha(-\psi+\Lambda)^{\alpha-1}\partial_0f(-\partial_t\psi)+\eps\frac{\partial f}{\partial h_{ij}}(-1)[h_{\varphi}]^i_j+\eps\frac{\partial f}{\partial h_{ii}}\\
&+\alpha (-\psi+\Lambda)^{\alpha-1}\frac{\partial f}{\partial h_{ij}}g^{i\bar{k}}(\omega_{\psi})_{j\bar{k}}-\alpha(-\psi+\Lambda)^{\alpha-1}\frac{\partial f}{\partial h_{ii}}.
\end{split}
\end{equation*}

Because of our choice of $\eps$ and $\Lambda$,  we have that 
\begin{equation}\label{3.6NN}
\frac{\eps}{2} \ge \alpha\Lambda^{\alpha-1}.
\end{equation}
So that in the above,  
\begin{equation*}
\eps\frac{\partial f}{\partial h_{ii}}-\alpha(-\psi+\Lambda)^{\alpha-1}\frac{\partial f}{\partial h_{ii}}>\frac{\eps}{2}\frac{\partial f}{\partial h_{i\bar{i}}}.
\end{equation*}
So that we get
\begin{equation*}
\begin{split}
&L\rho\ge -\eps \partial_0f(-\partial_t\varphi)+\eps \partial_0f+\alpha(-\psi+\Lambda)^{\alpha-1}\partial_0f(-\partial_t\psi)\\
&+\eps\frac{\partial f}{\partial h_{ij}}(-1)[h_{\varphi}]_j^i+\alpha(-\psi+\Lambda)^{\alpha-1}\frac{\partial f}{\partial h_{ij}}g^{i\bar{k}}(\omega_{\psi})_{j\bar{k}}+\frac{\eps}{2}\frac{\partial f}{\partial h_{ii}}.
\end{split}
\end{equation*}

Assume that $\rho$ achieves positive maximum on $[0,T]\times M$ at $(t_0,x_0)$.  We may choose a neighborhood of $x_0$ such that it is equivalent to a ball $B_{r_0}(x_0)$ under local coordinates.  We choose $\theta=\min(1,\frac{\eps r_0^2}{400})$,  and let $\eta$ be a cut-off function on $M$ satisfying the following conditions:
\begin{equation*}
\eta(x_0)=1,\,\,\eta\equiv 1-\theta,\,\textrm{on $M-B_{r_0}(x_0)$},\,\,|\nabla\eta|\le \frac{10\theta}{r_0},\,\,|D^2\eta|\le \frac{10\theta}{r_0^2}.
\end{equation*}
Then we can compute
\begin{equation*}
\begin{split}
&L(e^{\rho}\eta)=e^{\rho}\eta L\rho+e^{\rho}\eta\frac{\partial f}{\partial h_{ij}}g^{i\bar{k}}\rho_j\rho_{\bar{k}}+e^{\rho}L\eta+2e^{\rho}Re\big(\frac{\partial f}{\partial h_{ij}}g^{i\bar{k}}\rho_j\eta_{\bar{k}}\big)\\
&\ge -\eps e^{\rho}\eta(\partial_0f(-\partial_t\varphi)+\frac{\partial f}{\partial h_{ij}}[h_{\varphi}]_j^i)+e^{\rho}\eps \partial_0f\eta\\
&+e^{\rho}\alpha(-\psi+\Lambda)^{\alpha-1}\big(\partial_0f(-\partial_t\psi)+\frac{\partial f}{\partial h_{ij}}g^{i\bar{k}}(\omega_{\psi})_{j\bar{k}}\big)\eta+e^{\rho}\frac{\eps}{2}\frac{\partial f}{\partial h_{ii}}\eta\\
&+e^{\rho}L\eta+e^{\rho}\eta\frac{\partial f}{\partial h_{ij}}g^{i\bar{k}}\rho_j\rho_{\bar{k}}+2e^{\rho}Re(\frac{\partial f}{\partial h_{ij}}g^{i\bar{k}}\rho_j\eta_{\bar{k}}).
\end{split}
\end{equation*}
In the last term above,  we can estimate:
\begin{equation*}
2e^{\rho}Re\big(\frac{\partial f}{\partial h_{ij}}g^{i\bar{k}}\rho_j\eta_{\bar{k}}\big)\ge -e^{\rho}\frac{\partial f}{\partial h_{ij}}g^{i\bar{k}}\rho_j\rho_{\bar{k}}-e^{\rho}\frac{\partial f}{\partial h_{ij}}g^{i\bar{k}}\eta_j\eta_{\bar{k}}\ge -e^{\rho}\frac{\partial f}{\partial h_{ij}}g^{i\bar{k}}\rho_j\rho_{\bar{k}}-\frac{100\theta^2}{r_0^2}e^{\rho}\frac{\partial f}{\partial h_{ii}}.
\end{equation*}

Also 
\begin{equation*}
e^{\rho}L\eta\ge -e^{\rho}\frac{\partial f}{\partial h_{ii}}\frac{10\theta}{r_0^2}.
\end{equation*}
Therefore
\begin{equation*}
\begin{split}
&L(e^{\rho}\eta)\ge -\eps e^{\rho}\eta\big(\partial_0f(-\partial_t\varphi)+\frac{\partial f}{\partial h_{ij}}(\omega_{\varphi})_{i\bar{j}}\big)\\
&+\eps e^{\rho}\partial_0f\eta+e^{\rho}\alpha(-\psi+\Lambda)^{\alpha-1}(\partial_0f(-\partial_t\psi)+\frac{\partial f}{\partial h_{ij}}(\omega_{\psi})_{i\bar{j}})\eta\\
&+e^{\rho}\frac{\partial f}{\partial h_{ij}}g_{i\bar{j}}\big(\frac{\eps}{2}-\frac{100\theta}{r_0^2}-\frac{100\theta^2}{r_0^2}\big)\eta.
\end{split}
\end{equation*}
Because of our choice of $\eps$ and $\theta$,  we have that
\begin{equation}\label{3.7NN}
\frac{\eps}{2}-\frac{100\theta}{r_0^2}-\frac{100\theta^2}{r_0^2}>0.
\end{equation}
Therefore
\begin{equation}\label{3.6}
\begin{split}
&L(e^{\rho}\eta)\ge e^{\rho}\eta\big(-\eps\partial_0f(-\partial_t\varphi)-\eps\frac{\partial f}{\partial h_{ij}}[h_{\varphi}]^i_j+\eps\partial_0f\\
&+\alpha(-\psi+\Lambda)^{\alpha-1}(\partial_0f(-\partial_t\psi)+\frac{\partial f}{\partial h_{ij}}g^{i\bar{k}}(\omega_{\psi})_{j\bar{k}})\big).
\end{split}
\end{equation}
If we put $u=e^{\rho}\eta$ and denote the right hand side of (\ref{3.6}) to be $R$,  then (\ref{3.6}) is equivalent to:
\begin{equation*}
-\partial_tu+\frac{1}{\partial_0f}\frac{\partial f}{\partial h_{ij}}\partial_{i\bar{j}}u\ge \frac{1}{\partial_0f}R.
\end{equation*}
We wish to apply the parabolic Alexandrov maximum principle (Lemma ?) in $[0,T]\times B_{r_0}(x_0)$ so that 
\begin{equation}\label{3.7}
\sup_{[0,T]\times B_{r_0}(x_0)}e^{\rho}\eta\le \max\big(\sup_{[0,T]\times \partial B_{r_0}(x_0)}e^{\rho}\eta,\sup_{\{0\}\times B_{r_0}(x_0)}e^{\rho}\eta\big)+C_nr_0\bigg(\int_{E}\frac{(\frac{1}{\partial_0f} R^-)^{2n+1}}{(\det(\frac{1}{\partial_0f}\frac{\partial f}{\partial h_{ij}}))^2}\omega_0^ndt\bigg)^{\frac{1}{2n+1}}.
\end{equation}
In the above,  $R^-=\max(R,0)$,  $E$ is the set of $(t,x)$ on which $\partial_tu\ge 0$,  $D^2u\le 0$.  
Without loss of generality,  we may assume that 
\begin{equation}\label{3.10New}
\sup_{[0,T]\times M}\rho=K>0.
\end{equation}
Also from our choice of $x_0$,  we would have $\sup_{[0,T]\times B_{r_0}(x_0)}e^{\rho}\eta\ge e^K$.  On the right hand side,  we have
\begin{equation}\label{3.11New}
\sup_{[0,T]\times \partial B_{r_0}(x_0)}e^{\rho}\eta\le (1-\theta)e^K,\,\,\,\sup_{\{0\}\times B_{r_0}(x_0)}e^{\rho}\eta\le 1.
\end{equation}
The second inequality is due to that,  according to our choice of $s$,  we have $\rho\le 0$ for $t=0$.
It only remains to estimate the integral in (\ref{3.7}).  We claim that:
\begin{claim}\label{claim3.7}
On the set $E$,  we have that $\partial_t\varphi< 0$,  $\omega_0+\sqrt{-1}\partial\bar{\partial}\varphi>0$.  In particular,  we have that on $E$,  $(-\partial_t\varphi,\lambda[h_{\varphi}])\in \Gamma_+$.  (the positive cone)
\end{claim}
Indeed,  on the set $E$,  we have that $\partial_tu\ge 0$,  $u_{i\bar{j}}\le 0$,  which translates to
\begin{equation*}
0\le \partial_tu=e^{\rho}\partial_t\rho\eta=e^{\rho}(-\eps-\eps\partial_t\varphi+\alpha(-\psi+\Lambda)^{\alpha-1}\partial_t\psi),
\end{equation*}
Note that $\partial_t\psi\le 0$,  we immediately have $\partial_t\varphi<0$.
On the other hand,
\begin{equation*}
\begin{split}
&0\ge u_{i\bar{j}}=e^{\rho}\rho_{i\bar{j}}+e^{\rho}\rho_i\rho_{\bar{j}}+e^{\rho}\eta_{i\bar{j}}+e^{\rho}(\rho_i\eta_{\bar{j}}+\rho_{\bar{j}}\eta_i)\\
&\ge e^{\rho}\big(-\eps\varphi_{i\bar{j}}+\alpha(-\psi+\Lambda)^{1-\alpha}\psi_{i\bar{j}}+\alpha(1-\alpha)(-\psi+\Lambda)^{\alpha-2}\psi_i\psi_{\bar{j}}\big)\\
&+e^{\rho}\rho_i\rho_{\bar{j}}-e^{\rho}\frac{100\theta}{r_0^2}g_{i\bar{j}}-e^{\rho}\rho_i\rho_{\bar{j}}-e^{\rho}\eta_i\eta_{\bar{j}}\\
&\ge e^{\rho}\big(-\eps(g_{i\bar{j}}+\varphi_{i\bar{j}})+\alpha(-\psi+\Lambda)^{\alpha-1}(g_{i\bar{j}}+\psi_{i\bar{j}})+
(\eps-\alpha(-\psi+\Lambda)^{\alpha-1}\\
&-\frac{100\theta}{r_0^2}-\frac{100\theta^2}{r_0^2})g_{i\bar{j}}\big).
\end{split}
\end{equation*}
Note that because of (\ref{3.6NN}) and (\ref{3.7NN}),  we have that $\eps-\alpha(-\psi+\Lambda)^{\alpha-1}-\frac{100\theta}{r_0^2}-\frac{100\theta^2}{r_0^2}>0$.  Also we have $g_{i\bar{j}}+\psi_{i\bar{j}}>0$,  therefore $g_{i\bar{j}}+\varphi_{i\bar{j}}>0$ and the claim is proved.

The above Claim \ref{claim3.7} guarantees that we can use structural assumption (3) and (4) to estimate the integral on the right hand side of (\ref{3.7}). 
We first have that
\begin{equation}\label{3.10NNN}
\int_E\frac{(\frac{1}{\partial_0f}R^-)^{2n+1}}{(\det(\frac{1}{\partial_0f}\frac{\partial f}{\partial h_{ij}}))^2}\omega_0^ndt=\int_{E}\frac{\partial_0f(R^-)^{2n+1}}{(\partial_0f\det(\frac{\partial f}{\partial h_{ij}}))^2}\omega_0^ndt\le \int_{\{R<0\}\cap E}c_0^{-2}\partial_0f(R^-)^{2n+1}\omega_0^ndt.
\end{equation}
Now we go back to (\ref{3.6}),  and found that for $(t,x)\in E$,  we have,  using the structural assumption on $f$:
\begin{equation*}
\partial_0f(-\partial_t\varphi)+\frac{\partial f}{\partial h_{ij}}[h_{\varphi}]_i^j\le C_0f.
\end{equation*}
Also from Arithemetic-Geometric Inequality:
\begin{equation*}
\begin{split}
&\partial_0f(-\partial_t\psi)+\frac{\partial f}{\partial h_{ij}}g^{i\bar{k}}(\omega_{\psi})_{j\bar{k}}\ge (n+1)\big(\partial_0f\det(\frac{\partial f}{\partial h_{ij}})(-\partial_t\psi)\frac{\omega_{\psi}^n}{\omega_0^n}\big)^{\frac{1}{n+1}}\\
&\ge (n+1)\big(c_0Ent_p^{-1}(F)e^{(n+1)F}(F^2+1)^{\frac{p}{2}}\big)^{\frac{1}{n+1}}.
\end{split}
\end{equation*}
Therefore,  on the set $E$,  the right hand side of (\ref{3.6}) satisfies 
\begin{equation}\label{3.10NN}
R\ge e^{\rho}\eta f\big(-\eps C_0+\eps\frac{\partial_0f}{f}+\alpha(n+1)c_0^{\frac{1}{n+1}}Ent_p^{-\frac{1}{n+1}}(F)(F^2+1)^{\frac{p}{2(n+1)}}(-\psi+\Lambda)^{\alpha-1}\big).
\end{equation}
Therefore,  on the set $E\cap\{R<0\}$,  we have that
\begin{equation*}
\begin{split}
&\partial_0f\le C_0f,\,\,\,(F^2+1)^{\frac{p}{2(n+1)}}\le \eps C_0\alpha^{-1}(n+1)^{-1}c_0^{-\frac{1}{n+1}}Ent_p^{\frac{1}{n+1}}(F)(-\psi+\Lambda)^{1-\alpha},\\
&R^-\le e^{\rho}\eta f(\eps C_0).
\end{split}
\end{equation*}
Therefore
\begin{equation}\label{3.11NN}
f=e^F\le \exp\big((\eps C_0\alpha^{-1}(n+1)^{-1})^{\frac{n+1}{p}}c_0^{-\frac{1}{p}}Ent_p^{\frac{1}{p}}(F)(-\psi+\Lambda)^{\frac{(1-\alpha)(n+1)}{p}}\big).
\end{equation}
Combinning (\ref{3.10NN}) and (\ref{3.11NN}),  we may continue the estimate in (\ref{3.10NNN}):
\begin{equation*}
\begin{split}
&\int_{\{R<0\}\cap E}c_0^{-2}\partial_0f(R^-)^{2n+1}\omega_0^ndt\le \int_{[0,T]\times M}c_0^{-2}C_0f(\eps C_0)^{2n+1}e^{(2n+1)\rho}f^{2n+1}\omega_0^ndt\\
&\le \int_{[0,T]\times M}c_0^{-2}\eps^{2n+1}C_0^{2n+2}e^{(2n+1)\rho}\\
&\times\exp\big((2n+2)(\eps C_0\alpha^{-1}(n+1)^{-1})^{\frac{n+1}{p}}c_0^{-\frac{1}{p}}Ent_p^{\frac{1}{p}}(F)(-\psi+\Lambda)^{\frac{(1-\alpha)(n+1)}{p}}\big)\omega_0^ndt.
\end{split}
\end{equation*}
Note that because of our choice of $\eps$,  we have that
\begin{equation*}
(2n+2)(\eps C_0\alpha^{-1}(n+1)^{-1})^{\frac{n+1}{p}}c_0^{-\frac{1}{p}}Ent_p^{\frac{1}{p}}(F)\le \alpha_0.
\end{equation*}
Here $\alpha_0$ is the $\alpha$-invariant of the class $[\omega_0]$.  Therefore
\begin{equation}\label{3.15New}
\begin{split}
&\int_{\{R<0\}\cap E}c_0^{-2}\partial_0f(R^-)^{2n+1}\omega_0^ndt\\
&\le e^{(2n+1)K}\int_{[0,T]\times M}c_0^{-2}\eps^{2n+1}C_0^{2n+2}\exp\big(\alpha_0(-\psi+\Lambda)^{\frac{(1-\alpha)(n+1)}{p}}\big)\omega_0^ndt\\
&\le e^{(2n+1)K}c_0^{-2}\eps^{2n+1}C_0^{2n+2}\int_{[0,T]\times M}\exp\big(-\alpha_0\psi+\alpha_0\Lambda)\omega_0^ndt\\
&\le e^{(2n+1)K}c_0^{-2}\eps^{2n+1}C_0^{2n+2}e^{\alpha_0\Lambda}T\sup_{t\in[0,T]}\int_Me^{-\alpha_0\psi}\omega_0^n.
\end{split}
\end{equation}
We have seen in (\ref{3.5NNew}) that
\begin{equation*}
\sup_{t\in[0,T]}\int_Me^{-\alpha_0\psi}\omega_0^n\le C_*.
\end{equation*}
Here $C_*$ depends only on the background metric.
So that after combinning (\ref{3.7}),  (\ref{3.10New}),  (\ref{3.11New}),  (\ref{3.15New}),  we obtain that
\begin{equation}\label{3.16N}
e^K\le \max\big((1-\theta)e^K,e^{||\rho_0||_{L^{\infty}}}\big)+T^{\frac{1}{2n+1}}C_nr_0e^K\eps(c_0^{-2}C_0^{2n+2}e^{\alpha_0\Lambda}C_*)^{\frac{1}{2n+1}}.
\end{equation}
Therefore,  if we choose $T$ small enough so that
\begin{equation}\label{3.17N}
T^{\frac{1}{2n+1}}C_nr_0\eps(c_0^{-2}C_0^{2n+2}e^{\alpha_0\Lambda}C_*)^{\frac{1}{2n+1}}=\frac{\theta}{2},
\end{equation}
we would have that the max in (\ref{3.16N}) is achieved by $e^{||\rho_0||_{L^{\infty}}}$,  which gives us that $K\le ||\rho_0||_{L^{\infty}}.$
Denote the $T$ given by (\ref{3.17N}) to be $T_0$,  then we get an estimate for $\sup_{t\in[0,T_0]}\rho$.
Note that the $T_0$ given by (\ref{3.17N}) depends only on dimension,  the background metric,  and the structural constants $c_0$ and $C_0$,  we may repeat the same argument to estimate $\sup_{t\in[T_0,2T_0]}\rho$ and so on. 
\end{proof}
Here we state without proof the following parabolic Alexandrove maximum principle we used in the above proof (see \cite{Lieberman},  Theorem 7.1):
\begin{lem}
Let $\Omega\subset \mathbb{R}^n$ be a bounded domain and $u$ be a smooth function on $[0,T]\times \Omega$,  such that
\begin{equation*}
-\partial_tu+a^{ij}\partial_{ij}u\ge f.
\end{equation*}
Then 
\begin{equation*}
\sup_{[0,T]\times \Omega}u\le \sup_{\partial_P([0,T]\times \Omega)}u+C_n(\diam \Omega)^{\frac{n}{n+1}}\bigg(\int_E\frac{(f^-)^{n+1}}{\det a^{ij}}dxdt\bigg)^{\frac{1}{n+1}}.
\end{equation*}
In the above,  $\partial_P([0,T]\times \Omega)=(\{0\}\times \Omega)\cup([0,T]\times \partial\Omega)$,  $f^-=\max(-f,0)$ and
$E$ is the set of $(t,x)$ on which $\partial_tu\ge 0$ and $D_x^2u\le 0$.
\end{lem}

So we have shown that,  the right hand side of the estimate in Proposition \ref{p3.4} is uniformly bounded with the said dependence.  Then we can proceed in the same way as Section 2 to obtain the $L^{\infty}$ bound (note that from Lemma \ref{l2.7New} and on,  we only used the Moser-Trudinger inequality but not the equation).

\end{document}